\documentclass{article}
\usepackage[utf8]{inputenc}

\usepackage{amsmath}
\usepackage{amssymb}
\usepackage{hyperref}
\usepackage{graphicx}
\usepackage{bm}
\usepackage[framemethod=tikz]{mdframed}
\usepackage{enumerate}

\newenvironment{proof}{\noindent{\bf Proof}}{\hspace*{\fill}$\Box$}
\newenvironment{proofof}[1]{%
\noindent {\bf Proof of #1}}%
{\hspace*{\fill}$\Box$}

\newtheorem{theorem}{Theorem}[section]
\newtheorem{lemma} [theorem] {Lemma}
\newtheorem{corollary} [theorem] {Corollary}
\newtheorem{remark} [theorem] {Remark}
\newtheorem{conjecture} [theorem] {Conjecture}

\def\E{{\mathbb E}\,}

\newcommand{\pr}{\mathbb P}

\newcommand{\conc}{\mathcal{Q}}
\newcommand{\concdiam}{\tilde{\mathcal{Q}}}
\newcommand{\diam}{\mathrm{diam}}
\newcommand{\cl}{\mathrm{cl}}

\newcommand\nuopt[1]{\nu_{#1}^*}

\begin{document}

\title{Anticoncentration of random vectors via the strong perfect graph theorem}

\author{Tomas Juškevičius\footnote{Supported by the Czech Science Foundation, grant number 20-27757Y, with institutional support RVO:67985807.}\, and Valentas Kurauskas}

\date{12 August 2023}

\maketitle

\begin{abstract}
    In this paper we give anticoncentration bounds for sums of independent random vectors in finite-dimensional vector spaces. In particular, we asymptotically establish a conjecture of Leader and Radcliffe (1994) and a question of 
    Jones (1978). 

    The highlight of this work is an application of the strong perfect graph theorem by Chudnovsky, Robertson, Seymour and Thomas (2003) in the context of anticoncentration.
\end{abstract}

\bigskip

\hskip -0.5 cm 
\textbf{keywords}: concentration function; Littlewood-Offord problem; perfect graph

\vskip 1 cm

\maketitle

\section {Introduction}
\label{sec.introduction}

Let $X$ be a random vector taking values in some metric space $\mathcal{S}$. Given $t>0$ define  
$$\conc(X, t)=\sup_B \pr(X \in B)\text{\,\,\,\,and\,\,\,\,} \concdiam(X,t) = \sup_S \pr(X \in S),$$ 
where $B$ is an open ball and $S$ an open set, both having diameter at most $t$. We clearly have $\conc(X, t) \le \concdiam(X,t)$.

$\conc(X, \cdot)$ is often referred to as the \textit{L\'{e}vy concentration function} of $X$ or just the \textit{concentration function} of $X$ in the literature (sometimes the ball $B$ is considered to be closed), see e.g. \cite{Levy}, \cite{Bobkov}. The similar function $\concdiam(X, \cdot)$ was used by Leader and Radcliffe~\cite{LR}
in their work on generalized Littlewood-Offord type inequalities.


 One special case of our results (see Section~\ref{sec.results}) is as follows.
\begin{theorem}\label{thm.front}
    Fix an integer $d \ge 2$.
    Let $X_1, \dots, X_n$ be independent random vectors in $\mathbb{R}^d$ equipped with a norm $\|\cdot\|$. Assume that for some $k\in \mathbb{N}$ and each $i \in \{1, \dots, n\}$ we have 
    $$\concdiam(X_i, 1) \le \frac{1}{k}.$$
Let $U_1,\ldots, U_n$ be 
    i.i.d. random variables distributed uniformly on the set $\{ 1,2,\ldots,k\}$.
    Then
    \begin{eqnarray}
\concdiam(X_1 + \dots + X_n, 1) &\le& 
        (1 + o(1)) \concdiam\left(U_1 + \dots + U_n, 1\right)
        \label{eq.mainfront}
        \\&=& 
        (1 + o(1))\max_{\ell \in \mathbb{Z}}\mathbb{P}\left(U_1 + \dots + U_n=\ell\right).
        \nonumber
    \end{eqnarray}
\end{theorem}
Note that even though $\conc$ and $\concdiam$ are defined using diameter, which depends
on the distance or the norm, our convention will be not to indicate 
this in the notation. For 1-dimensional random variables like $U_i$, 
the usual distance $d(x,y) = |x-y|$ will be assumed, unless mentioned otherwise. For claims about general
normed spaces with norm $\|\cdot\|$, like on the left of (\ref{eq.mainfront}),
we implicitly use the induced distance $d(x,y) = \|x-y\|$. 
Let us also note that the $o(1)$ term above is dependent on $k$, the norm $\|\cdot\|$ and the dimension $d$, which is not reflected in the notation. 

The inequality 
\eqref{eq.mainfront} is clearly optimal apart from the multiplicative factor $1+o(1)$. This is because $\concdiam(U_i,1)=\frac{1}{k}$ and we can attain the value $\concdiam(U_1+\ldots+U_n, 1)$ by $X_i$ that 
are uniformly distributed on the first $k$ integer multiples 
for some unit vector $v \in \mathbb{R}^d$. 

The case $k=2$ in Theorem~\ref{thm.front} corresponds to the classical Littlewood-Offord problem.
Leader and Radcliffe~\cite{LR} conjectured that in this case the inequality \eqref{eq.mainfront} holds in arbitrary normed spaces without the multiplicative error factor (Conjecture 13, p. 101). 
They proved it for $d=1$. 
Jones~\cite{jones} asked whether or not the inequality \eqref{eq.mainfront} without the error factor holds when $k>3$ 
and each random variable $X_i$ has a uniform distribution on a $k$-point set $A_i$
in a Banach space such that the pairwise distances in $A_i$ are at least $1$ (Question 1 in Section~4). The latter is a special case of our setup. Consequently, Theorem~\ref{thm.front} establishes both problems asymptotically in finite-dimensional spaces.

The highlight of our work is the use of the strong perfect graph theorem that allows us to represent distributions of random vectors in a useful way so that existing combinatorial methods can be applied to the components of the decomposition. We believe that our approach can be useful in other contexts.

The paper is organized as follows. Notation and important definitions are presented and main results are formulated in Section~\ref{sec.results}. We outline the proof strategy in Section~\ref{sec.outline}. The subsequent Sections~\ref{sec.discrete}--\ref{sec.general} deal with different parts of the proof. 

\section{Main results}
\label{sec.results}

For a finite multiset $A$ let $\nu(A)$ be the uniform measure on~$A$. Thus, for example, if $A=\{0,0,1\}$ then $\nu(A)$ is the distribution of a Bernoulli random variable with success probability $p=\frac{1}{3}$. We call a finite set of $k$ points in a metric space a $k$-\emph{block} if those points are at pairwise distance at least $1$. We shall sometimes refer to the latter sets as \emph{blocks} without specifying the particular parameter $k$. 

Let us define a parameterized family of probability measures that will serve as worst-case distributions in the forthcoming results. For $\alpha \in (0,1]$ set $k=\lfloor \alpha^{-1} \rfloor$ and take the unique $p=p(\alpha) \in (0,1]$
such that $\frac p k + \frac {1-p} {k+1} = \alpha$. Define a  measure on $\mathbb{R}$ by
\begin{equation}\label{eq.nu_star}
    \nuopt \alpha :=  p \nu(\{1,2,\dots,k\} - \tfrac {k+1} 2) + (1-p) \nu(\{1,2,\dots, k+1\} - \tfrac {k+2} 2).
\end{equation}

Note that in the special case $\alpha=\frac{1}{k}$ we have $p=1$ and so in view of \eqref{eq.nu_star} the measure $\nuopt \alpha$ is the uniform distribution on a symmetric arithmetic progression on $k$ points with difference $1$. It therefore has the same distribution as $U_i-\E U_i$ with $U_i$ as in Theorem~\ref{thm.front}. When $\alpha$ is not an inverse of an integer the measure $\nu^*_\alpha$ is a non-trivial convex combination of two distinct uniform distributions.

We can now state a more general version of Theorem~\ref{thm.front}, which is one of our main results.

\begin{theorem}\label{teor1}
    Fix an integer $d \ge 2$, and $\alpha \in (0,1]$.
    Let $X_1, \dots, X_n$ be independent random vectors in $\mathbb{R}^d$ endowed with a norm $||\cdot||$. Let $Y_1,\ldots, Y_n$ be i.i.d. random variables with distribution $\nuopt \alpha$. Assume that for each $i$
    \begin{align}\label{cond1}
    \concdiam(X_i, 1) \le \alpha.
    \end{align}
    We then have
    \begin{align}\label{eq.thm1}
        \concdiam(X_1 + \dots + X_n, 1) \le (1 + o(1)) \concdiam(Y_1 + \dots + Y_n, 1),
    \end{align}
    where the $o(1)$ term depends only on $\alpha$, $d$ and the underlying norm $||\cdot||$. 
\end{theorem}

As mentioned above, in the one-dimensional case the sharp inequality (without the $1+o(1)$ multiplicative factor) was proved by Leader and Radcliffe~\cite{LR} when $\alpha$ is an inverse of an integer. Answering a question of theirs it was extended to $d=1$ and all values of $\alpha$ by the first author in \cite{disertacija} and \cite{tj}.

There are essential differences between the one-dimensional and the higher-dimensional case. This is due to the fact that the class of probability distributions satisfying the condition \eqref{cond1} when $d>1$ is in some sense much richer in contrast to the case $d=1$, see Section~\ref{sec.outline}. 
The crucial tool to analyze the situation in high dimensions, perhaps surprisingly, comes from structural graph theory. To be more precise, we use the strong perfect graph theorem famously proved by Chudnovsky, Seymour, Robertson and Thomas~\cite{strong_perfect}. We discuss this in greater detail in Section~\ref{sec.outline}.

Furthermore, our methods allow us to obtain the optimal inequality in greater generality provided all distributions are close to a single line.

Let $\mathcal{S}$ be a normed space with norm $\|\cdot\|$. The norm induces a distance $d(x,y) = \|x-y\|$. For $x \in \mathcal{S}$ and $S$ a set in $\mathcal{S}$ we write $d(x, S) = \inf_{y \in S} d(x,y)$.  A line in $\mathcal{S}$ is a set $\{t \upsilon + b: t \in\mathbb{R}\}$ where $\upsilon,b \in \mathcal{S}$ and $\|\upsilon\|=1$. 

\begin{theorem}\label{close_to_line}
    Let $X_1, \dots, X_n$ be independent random vectors with values in a normed space over the reals $\mathcal{S}$ with norm $\|\cdot\|$
    and let $\alpha_1, \dots, \alpha_n \in (0, 1]$.
    Assume that $\concdiam(X_1,1) \le \alpha_1, \dots, \concdiam(X_n, 1) \le \alpha_n$ and that there is a line $L$ in $\mathcal{S}$ so that $\sup_{\omega} d(X_i(\omega),L) < \frac 1 8$ ($ < \frac {\sqrt 3} 4$ if $\mathcal{S}$ is Hilbert) for each $i$. If either
    \begin{enumerate}[i)]
    \item the random vectors $X_i$ have countable supports; or
    \item $\mathcal{S}$ is complete and separable; 
    \end{enumerate}
   we have that
    \[
        \concdiam(X_1+\dots+X_n, 1) \le \concdiam(Y_1 + \dots + Y_n, 1) = \pr(Y_1 + \dots + Y_n \in \{0, \frac 1 2\}),
    \]
    where $Y_1, \dots, Y_n$ are independent and $Y_i$ has distribution $\nuopt {\alpha_i}$ defined in (\ref{eq.nu_star}).
\end{theorem}
The latter inequality is optimal and in the case $\alpha_i=\frac{1}{2}$ it gives the desired result in the Leader-Radcliffe conjecture and when $\alpha_i=\frac 1 k$ it also answers the discussed question by Jones in this restricted setting. 

Finally, our methods allow us to treat the 
general situation with a non-uniform condition in \eqref{cond1} 
when a certain kind of a local limit theorem for the extremal distributions holds.
We shall formulate it in the language of triangular arrays. 
The interested reader can find a general non-asymptotic version of this result in Section~\ref{sec.general}, see Theorem~\ref{thm.main}.
\begin{theorem}\label{local}
    Fix an integer $d \ge 2$, and a norm $\|\cdot\|$ on $\mathbb{R}^d$. Then there is a constant $C = C(d, \|\cdot\|)$ such that the following holds.

    For each $l \in \{1,2,\dots,\}$ let  $X_{l, 1}, \dots, X_{l, n_l}$
    be a sequence of independent random variables in $(\mathbb{R}^d, \|\cdot\|)$ such that $\concdiam(X_{l, 1}, 1) \le \alpha_{l, 1}, \dots, \concdiam(X_{l, n_l}, 1) \le \alpha_{l, n_l}$ and  $1 \ge \alpha_{l,n_l} \ge \dots \ge \alpha_{l, 1} > 0$.
    Let $Y_{l, 1}, \dots, Y_{l, n_l}$ be independent random variables with distributions $\nuopt {\alpha_{l, 1}}$, $\dots$, $\nuopt {\alpha_{l, n_l}}$ respectively.
    Define a function $\xi_{d}:\, \mathbb{R}\rightarrow \mathbb{R}$ by setting $\xi_{d}(x)=x$ if $d=2$ and $\xi_{d}(x)=1$ otherwise. Denote
    \[
        \bar{\alpha}_l = n_l^{-1} \sum_{i=1}^{n_l} \alpha_{l, i} \quad \mbox{and}\quad V^*_{l, t} = \sum_{i=1}^{t} \E Y_{l, i}^2.
    \]
    Suppose  $V^*_{l, n_l} \to \infty \mbox{ as }l\to \infty$, 
    \begin{align*}
        &\xi_d(\bar{\alpha}_l)^2 V^*_{l, n_l} = o(n_l^2),
        & &V^*_{l, n_l} \le e^{\frac {C^2} {36} n_l^{\frac 1 2}},
 \\
        &  \sum_{i=1}^{n_l} \E |Y_{l, i}|^3 = o( (V^*_{l, n_l})^{\frac 3 2}), 
        & &
        \xi_d(\bar{\alpha}) \bar{\alpha}_l^2 n_l = o((V^*_{l, n_l})^{3/2})
    \end{align*}
    and
    \[
        \lim_{\epsilon \downarrow 0} \lim_{l \to \infty} \frac {V^*_{l, \lceil n_l (1-\epsilon) \rceil}}{V^*_{l,n_l}} = 1.
    \]
    Then
    \begin{align}
        \concdiam(X_{l, 1} + \dots + X_{l, n_l}, 1)  & \le \concdiam(Y_{l, 1} + \dots + Y_{l, n_l}, 1) (1 + o(1))  \nonumber
        \\ & = \frac 1 {\sqrt{2 \pi V^*_{l, n_l}}} (1 + o(1)) \label{eq.sharpb}
      \\  &\le 
        \sqrt {\frac 6 \pi}  \frac {\bar{\alpha}_l} {\sqrt{ (1 - \bar{\alpha}_l^2) n_l}} (1 + o(1)). \label{eq.crudeb}
    \end{align}

\end{theorem}


Notice that the bound in (\ref{eq.sharpb}) is like what we would expect from a Gaussian distribution with variance $V^*_{l, n_l}$. 
In some simple cases (\ref{eq.sharpb}) follows by a standard local limit theorem. However, when, for example, part of $\alpha_{l,i}^{-1}$ are odd integers and the remaining ones are even integers, the distribution of  $Y_{l, 1} + \dots + Y_{l, n_l}$ is a mixture of distributions on two different lattices and we need a different argument (see Lemma~\ref{lem.clt} below). 

Write $a_n \gg b_n$ if $b_n = o(a_n)$ as $n \to \infty$.
\begin{remark}\label{rmk.cormain} In the case of the uniform bound for all random variables, i.e. $\alpha_{i,l} = \bar{\alpha}_l$ for all $i\in\{1, \dots, n_l\}$ the conditions of Theorem~\ref{local} simplify to 
    \begin{equation*}\label{eq.rmkcormain}
        \bar{\alpha}_l \gg n_l^{- \frac 1 2} \quad \mbox{and} \quad 
        1 - \bar{\alpha}_l \gg n_l^{-\frac 1 3}.
    \end{equation*}
    (When $d=2$ the second condition can be weakened to $\bar{\alpha}_l \ge e^{- \frac {C^2} {71} \sqrt{n}}$.)
\end{remark}

It would be tempting to drop the $(1+o(1))$ factors and the assumptions of our results and 
formulate the following.
\begin{conjecture}\label{conj.general}
    Let $X_1, \dots, X_n$ be independent random vectors with values in a normed space $\mathcal{S}$.
    Suppose $\concdiam(X_1,1) \le \alpha_1, \dots, \concdiam(X_n, 1) \le \alpha_n$, where $\alpha_1, \dots, \alpha_n \in (0,1]$. Then
    \[
        \concdiam(X_1+\dots+X_n, 1) \le \pr(Y_1 + \dots + Y_n \in \{0, \frac 1 2\})
    \]
    where $Y_1, \dots, Y_n$ are independent and $Y_i$ has distribution $\nuopt {\alpha_i}$ defined in (\ref{eq.nu_star}).
\end{conjecture}
In particular, the case $\alpha_1=\dots=\alpha_n=\frac 1 2$ corresponds exactly to the conjecture by Leader and Radcliffe and $\alpha_1=\dots=\alpha_n=\frac 1 k$ is a generalization of the question of Jones. 

We now show that this natural conjecture is false, even for $n=2$ and $\alpha_1=\alpha_2$.
A counterexample for 2-dimensional Euclidean space is when both $X_1$ and $X_2$ are chosen uniformly at random from the vertices of a regular octagon with radius $(\sqrt{2} + 2)^{-\frac 1 2} \approx 0.5412$. Then $\conc(X_1 + X_2, 1) = \concdiam(X_1+X_2, 1) = \concdiam(X_1, 1) = \frac 3 8$. This value is a sharp upper bound when $\alpha_1=\alpha_2=\frac 3 8$ since $\concdiam(X_1+X_2,1)\leq \concdiam(X_1,1)$. It is also strictly larger than $\concdiam(Y_1+Y_2, 1)$ for $Y_1, Y_2 \sim \nuopt \alpha$ and $\alpha \in [\frac 3 8, \frac 5 {12})$. Non-discrete counterexamples may be obtained when $X_1$ and $X_2$ are uniformly distributed on a circle in $\mathbb{R}^2$ of an appropriate radius.

Although we have just demonstrated that the worst case $1$-dimensional examples do not in general maximize $\concdiam(\cdot, 1)$,
they are still optimal asymptotically by Theorem~\ref{local} at least when the conditions of that theorem hold.

We can trivially increase $n$ in our counterexamples by appending any number of $\alpha_i$s that are close to 1. But we have neither non-trivial constructions for $n>2$ nor ones where each $\alpha_i^{-1}$ is integer. So Conjecture~\ref{conj.general} might still be true when all $\alpha_i^{-1}$ are integers. 



\section{Difficulties, the approach and an outline of the proofs}
\label{sec.outline}

Let $X_{1},..., X_{n}$ be independent random vectors in a normed space $\mathcal{S}$ equipped with a norm $||\cdot||$. Suppose that
\begin{equation}\label{concentr}
\concdiam (X_i, 1)\leq \alpha_{i}, \quad i \in \{1, \dots, n\}.
\end{equation}
The first typical step in doing optimization with such random variables in our context is passing to a much smaller class of distributions. 
The probabilities of the form
$\mathbb{P}(X_1+\cdots+X_n\in A)$ are linear in each distribution. 
Thus one example of such a class would be the
extreme points of the set of measures satisfying (\ref{concentr}).
(Recall that a point is extreme if it
does
not lie in the interior of any segment contained in the set.) 
The extreme measures are often quite simple or are contained in another simple set. If we can prove a desired inequality 
when each distribution is extreme, then the inequality holds for 
any distributions satisfying~(\ref{concentr}). 


As it was shown implicitly by Leader and Radcliffe~\cite{LR}, any distribution on $\mathbb{R}$ satisfying (\ref{concentr}) with $\alpha_i=\frac{1}{k}$ can be expressed as a convex combination of uniform distributions on  $k$-blocks. This then allowed them to pass to working with just such distributions and use tools from extremal combinatorics 
which originate 
from Erd\H{o}s's well-known approach \cite{ELO}.

Answering a question of Erd\H{o}s, Kleitman~\cite{Kleitman} established the sharp upper bound for $\concdiam$ for sums of random vectors each distributed on some $2$-block. Regrettably, we were unable to use the ideas of this short and simple proof, due to the fact that in some sense there are ``too many'' extreme  points and in some sense only the minority are distributions on $2$-blocks. It is not difficult to come up with examples of 
distributions in the 
plane
that are extreme in the set of measures satisfying (\ref{concentr}) with $\alpha_i=\frac{1}{2}$ and that cannot be written as convex combinations of distributions on $2$-blocks. Consider, for example, the uniform distribution on the vertices of an equilateral triangle with a side length of $1$ and the point at its center. It has concentration exactly $\frac{1}{2}$ and it is not difficult to check that it is extreme.

As it turns out, in the latter example it actually makes sense to decompose the measure into a delta measure on the center of the triangle and the uniform distribution on the three vertices of the triangle (a $3$-block). The first measure does not satisfy (\ref{concentr}) with $\alpha_i=\frac{1}{2}$, but the second distribution satisfies it with $\alpha_i=\frac{1}{3}$. In this work we show
how to decompose arbitrary measures into not necessarily equally sized blocks in an optimal way
so that traditional techniques for proving concentration inequalities can be applied.

We start by approximating the underlying distributions with discrete ones. We then examine the properties of their distance graphs (we join points of the support of the distribution if they are at a distance $<1$ in the underlying norm). We show that those distributions that are close to some line in $\mathcal{S}$ have distance graphs with no induced odd cycles of length $\geq 5$ and the same is true for the complements of such graphs. By the strong perfect graph theorem proved by Chudnovsky, Robertson, Seymour and Thomas~\cite{strong_perfect} such graphs are perfect, i.e., for all their induced subgraphs the chromatic number is equal to the clique number. 

The condition
(\ref{concentr}) provides a bound on a clique number of the distance graph. For instance, if the distribution under consideration is uniform and it satisfies (\ref{concentr}) with $\alpha_i=\frac{1}{2}$, it means that the distance graph of its support has no clique larger than half of the number of its vertices. A vertex-colouring of the distance graph 
provides a decomposition of the distribution into blocks where each block corresponds to a colour class. 
Since in our case the clique number and the chromatic number coincide by the strong perfect graph theorem, we obtain a decomposition of our distribution into the optimal number of blocks. 

This allows us to obtain optimal results in this restrictive setting for arbitrary normed spaces, not just $\mathbb{R}^d$ (see Theorem~\ref{close_to_line}). 

The remaining case is settled using a classical result by Hal{\'a}sz \cite{halasz}. We use it to show that if $\concdiam$ of our sum is at least the value we are trying to establish (which is of order $n^{-\frac 1 2}$ in the i.i.d. case), then most of the distributions can be made close to the same line by shifting.

The proofs are organized as follows. 
The relevant discrete approximation which allows us to pass to finitely supported distributions is provided in Section~\ref{sec.discrete}. We then turn to establishing properties of distance graphs of finite sets close to a line. For methodical reasons we treat the case of the Euclidean norm first in Section~\ref{sec.perfect} and a generalization is given in Section~\ref{sec.btk}.
In Section~\ref{sec.decomposition} we present our result on the decomposition of measures. 
In Section~\ref{sec.btk} we prove that blocks that are close to a line satisfy, the so called, B.T.K. chain condition, see \cite{btk, jones, LR}. In Section~\ref{sec.nice_proofs} we use this property in a similar way as Jones~\cite{jones} and our decomposition to prove Theorem~\ref{close_to_line}.
Next, we proceed to the general case of distributions that are not necessarily close to a line. The tools from the work of Hal{\'a}sz are presented and adapted to our setting in Section~\ref{sec.halasz}.
Finally, in Section~\ref{sec.general} we prove a version of local limit theorem, Lemma~\ref{lem.clt}, for our extremal distributions,
and combine all of the above to prove the main result, Theorem~\ref{thm.main}. Theorems~\ref{thm.front}, \ref{teor1} and \ref{local} are 
then obtained as corollaries of Theorem~\ref{thm.main}.

\section{Discrete approximation}
\label{sec.discrete}

A natural way to approximate an arbitrary probability distribution $\mu$ by a discrete one is to use
empirical distributions \cite{varadarajan}.
In general this is not enough to ensure continuity of the concentration function \cite{eddyhartigan}.
For example, let $X$ be distributed uniformly on a unit circle in the plane.
The empirical distribution $X_n$ of $n$ i.i.d copies of $X$ has $\concdiam(X_n, 2)=1$ almost surely,
while $\concdiam(X, 2) = \frac 1 2$. See \cite{eddyhartigan} for a more general result. For concentration
a general way out turns out to be sampling from $(1+\delta) X$ instead.


If a random vector $X$ has distribution $\mu$, it will be sometimes be more convenient to use the notation $\concdiam(\mu, \cdot)$ for $\concdiam(X, \cdot)$, and similarly for $\conc$. We write $\mu_n \Rightarrow \mu$ if a sequence of measures $(\mu_n)$ converges weakly to a measure $\mu$, see \cite{billingsley_1999}.


\begin{lemma}\label{lem.weak_discrete}
    Let $X$ be a random vector with values in a complete separable normed space ${\mathcal S}$.
    Suppose $\concdiam(X,1) = \alpha$.
    Fix $\delta > 0$ and let $\mu$ be the distribution of $(1+\delta)X$.  For each $n$ define the empirical measure $M_n=M_{\omega,n}$, such that given a realization $\omega$ of i.i.d. random vectors $\xi_1, \xi_2, \dots $ distributed according to $\mu$,  
    $M_{\omega, n}(A) = n^{-1} |\sum_{i=1}^n \mathbf{1}_A(\xi_i(\omega))|$ for each Borel set $A$.
    Then almost surely
    \begin{align}
        &M_n \Rightarrow \mu  \label{eq.weak_discrete}
           \intertext{and}
        &\limsup_{n \to \infty} \concdiam(M_n,1) \le \alpha. \label{eq.weak_diam}
    \end{align}
\end{lemma}

\begin{proof}
    Varadarajan~\cite{varadarajan} proved, see also Exercise~3.1 of Billingsley~\cite{billingsley_1999}, that (\ref{eq.weak_discrete}) holds almost surely.
   
    We will show that for each sequence of non-random measures $(M_n)$ (\ref{eq.weak_discrete}) implies (\ref{eq.weak_diam}).
    Let $Y_n$ be a random vector with distribution $M_n$.

    
    
    Suppose (\ref{eq.weak_diam}) does not hold, i.e., there is $\epsilon > 0$, an infinite subsequence $(n_k)$ and open sets $A_{n_k}$ of diameter at most 1 such that $\pr(Y_{n_k} \in A_{n_k}) > \alpha + \epsilon$. 

    Since $\mathcal{S}$ is separable and complete, by Prokhorov's theorem (see Theorem~5.2 of \cite{billingsley_1999}), $\{M_n, n\in\{1,2,\dots\}\}$ 
    is tight.
    That is, there exists a compact subset $K$ 
    such that $M_{n_k}(K) \ge 1 - \frac \epsilon 2$ for all $k \in \{1, 2, \dots\}$. So we have $\pr(Y_{n_k} \in K \cap A_{n_k}) > \alpha + \frac \epsilon 2$.

    Let $\diam(A)$ denote the diameter of a set $A\subseteq\mathcal{S}$ and let $\cl(A)$ denote its closure.
    Since $K$ is compact and complete, the space of closed bounded subsets of $K$ with Hausdorff distance is a compact and complete metric space (see, e.g., \cite{munkres}, p. 281). So the sequence $(n_k)$ 
    has a subsequence $(n_k')$ such that $\cl (K \cap A_{n_k'})$ converges in Hausdorff distance to a fixed set $A'$ as $k \to \infty$. As $\diam(\cl (K \cap A_{n_k})) = \diam(K \cap A_{n_k}) \le 1$ it follows that $A'$
    (and hence also each $\cl (K \cap A_{n_k'})$ for $k$ large enough) is contained in a closed set $A''$ with $\diam(A'') \le 1 + \frac \delta 2$.
    Thus for all $k$ large enough
    \begin{equation*} 
        \pr(Y_{n_k'} \in A'') > \alpha + \frac \epsilon 2.
    \end{equation*}
    Now using Theorem~2.1 of \cite{billingsley_1999} and (\ref{eq.weak_discrete}) we obtain a contradiction
    \[
        \limsup_{n\to \infty} \pr(Y_n \in A'') \le \pr((1+\delta) X \in A'') = \pr(X \in (1+\delta)^{-1} A'') \le \alpha.
    \]
    This is because $\diam( (1+\delta)^{-1} A'') \le (1+\delta/2) (1+\delta)^{-1} < 1$ so there is an open set of diameter at most 1
    that contains  $(1+\delta)^{-1} A''$. 
\end{proof}

\begin{lemma}\label{lem.discrete_approx_try}
    Let $X_1, \dots, X_m$ be independent random vectors with values in a complete separable normed space $\mathcal{S}$. Suppose 
    $\concdiam(X_i,1) \le \alpha$ and write $S = X_1 + \dots + X_m$. 
    Let $\epsilon > 0$. Then there is $\delta_0 > 0$
    such that for any $\delta \in (0, \delta_0)$ 
    there exist random vectors $\{Y_{i,n}, i \in \{1,\dots, m\}, n \in \{1, 2, \dots\}\}$ such that 
    $Y_{i,n}$ has the uniform distribution on a multiset of cardinality $n$,
    for each $n$ the vectors $Y_{1,n}, \dots, Y_{m,n}$ are independent,
    for each $i \in \{1,\dots,m\}$
    $Y_{i,n} \Rightarrow (1+\delta) X_i$, $\limsup_n \concdiam(Y_{i,n},1) \le \alpha$ as $n\to \infty$
    and for $T_n = Y_{1,n} + \dots + Y_{m,n}$ we have
    $\liminf_n \concdiam(T_n,1) \ge \concdiam(S,1) - \epsilon$
    as $n \to \infty$.
\end{lemma}

\begin{proof}
     Let $A$ be an open set of diameter at most
     1 in $\mathcal{S}$ such that $\pr(S \in A) \ge \concdiam(S, 1) - \frac{\epsilon} 2 > 0$.
%
%
     The convex hull of $A$ is also an open set with the same diameter,
     therefore we can assume $A$ is convex and
     $0 \in A$ (shift one of the random vectors if necessary).

     Consider open sets $A_k = (1 - k^{-1}) A$ for $k \in \{2, 3, \dots\}$. We have $\diam (A_k) = (1-k^{-1}) \diam(A)$
     and by continuity of measure (e.g., \cite{durrett}, Theorem 1.1.1) $\pr(S \in A_k) \to \pr(S \in A)$ as $k \to \infty$.
     Pick $k_0$ such that 
     $\pr(S \in A_{k_0}) \ge \concdiam(S,1) - \epsilon$. Set $\delta_0 := \frac 1 {k_0-1}$.

     Fix $\delta \in (0, \delta_0)$.
     For $i \in \{1, \dots, m\}$
     apply Lemma~\ref{lem.weak_discrete} with $\delta$ and $X_i$
     to obtain a sequence $(Y_{i,n})$ with non-random distributions $\mu_{i,n}$ such that $Y_{i,n} \Rightarrow (1+\delta) X_i$ with $\limsup_{n\to\infty}\concdiam(Y_{i,n},1) \le \alpha$.
     Note that $(1+\delta)^{-1} \ge 1-k_0^{-1}$.

	 As $Y_{i,n}$ converges in distribution to $(1+\delta) X_i$, 
     $\{Y_{i,n}, i \in \{1, \dots, m\}\}$ are independent for each $n$
     and $\{(1+\delta) X_i, i \in \{1,\dots,m\}\}$ are independent, it follows (see, e.g. Example~3.2 of \cite{billingsley_1999}) that $(Y_{1,n}, \dots, Y_{m,n})$ converges in distribution to $((1+\delta) X_1, \dots, (1+\delta)X_m)$, therefore using Theorem~2.1 of \cite{billingsley_1999} $T_n \Rightarrow (1 + \delta)S$.
%
%
     By the same theorem
     $\lim \inf \pr(T_n \in A) \ge \pr((1+\delta) S \in A) = \pr(S \in (1+\delta)^{-1} A) \ge \concdiam(S,1) - \epsilon$ since  $A_{k_0} = (1+\delta_0)^{-1} A \subseteq  (1+\delta)^{-1} A$.
\end{proof}

\section {Euclidean distance graphs for points close to a line}
\label{sec.perfect}

We define a \emph{distance graph} $G=G(x_1, \dots, x_N)$
of a sequence of $N$ points $x_1, \dots, x_N$ in a metric space $\mathcal{S}$ with distance $d$ by 
connecting two points if they are at distance less than~1:
\begin{align*}
    V(G) = \{x_1, \dots, x_N\}; & &  E(G) = \{x_i x_j : i \ne j \mbox { and } d(x_i,x_j) < 1\}.
\end{align*}
We allow $x_i = x_j$ for $i \ne j$, in this case vertices can be distinguished by their index,
and similarly for edges. So the graph in all cases is simple.

The distance in the above definition implicitly depends on the metric or the norm induced by the space $\mathcal S$; in this section we start with Euclidean spaces with ($l_2$) norm $\|\cdot\|_2$ and ($l_2$) distance $d_2(x,y) = \|x-y\|_2$. 

We will need the following simple fact.
\begin{lemma}\label{lem.pyth}
    Let $r=(r_1, \dots, r_d)^T$ and $s=(s_1, \dots, s_d)^T$ be
    two points in $\mathbb{R}^d$ such that $\|s-r\|_2 \ge 1$,
    and both $r$ and $s$ are at $l_2$ distance at most $\frac {\sqrt 3} 4$ from the 
    the line $L$ spanned by the vector $e^1 = (1,0, \dots, 0)^T$ (the ``$x$ axis'').
    Denote by $r^*=\langle r, e^1 \rangle e^1$ and $s^*=\langle s,e^1 \rangle e^1$ their projections
    onto $L$. Then
    \begin{equation*} 
        |r_1 - s_1| = \|r^*-s^*\|_2 \ge 
        \frac 1 2.
    \end{equation*}
\end{lemma}

\begin{proof}
    Let $q=s^*+r-r^*$. 
    Then $\langle s-q, r-q \rangle=\langle s-s^* - (r-r^*), r^*-s^* \rangle = 0$, so 
    by the Pythagorean theorem $\|r^*-s^*\|^2_2=\|r-q\|^2_2=\|s-r\|^2_2-\|q-s\|^2_2$. 
    Since $\|s-r\|_2 \ge 1$ 
    and by the triangle inequality and the assumption $\|q-s\|_2 = \|s^*-s +r-r^*\|_2 \le \|s-s^*\|_2 + \|r-r^*\|_2 \le \frac {\sqrt 3} 2$,
    we have $\|r^*-s^*\|^2_2 \ge 1 - (\frac{\sqrt 3} 2)^2$. 
\end{proof}

%
%

\begin{lemma}\label{lem.C5_complement}
    Let $G$ be a distance graph of points $x_1, \dots, x_N$ in $\mathbb{R}^d$,
    and assume each of the points is of distance at most $\frac {\sqrt 3} 4$
    from the line generated by the vector $(1, 0, \dots,0)^T$ (the ``$x$ axis''). 
    Then $\bar{G}$ does not contain
    an induced cycle of odd length $5$ or more.
\end{lemma}
\begin{proof}
    Suppose there exists such an induced cycle $C_k$ in $\bar{G}$.

    If there are $x_i, x_j$, $i \ne j$ with $x_i=x_j$ both on $C_k$,
    they cannot be neighbours on $C_k$ as $x_ix_j \in E(G)$.
    If they are non-neighbours on $C_k$ then since $k\ge 5$, there is a 
    neighbour $y$ of $x_i$ on $C_k$ which is not a neighbour of $x_j$ on $C_k$.
    So we have both $x_i y \not \in E(G)$ and  $x_j y \in E(G)$, a contradiction. 
    Thus we can assume each vertex in $C_k$ represents a distinct point.

    By Lemma~\ref{lem.pyth} the first coordinates
    for any neighbours on $C_k$ differ by at least $\frac 1 2$.

    Suppose first there is a vertex $v$ that has neighbours $u$ and $w$ on $C_k$
    such that $u_1 < v_1 < w_1$. Then by Lemma~\ref{lem.pyth} $\|u-w\|_2 \ge w_1 - u_1 = w_1 - v_1 + v_1 - u_1 \ge 1$,
    a contradiction, since $k \ge 5$,  $u$ and $w$ are non-neighbours on the cycle and so they must
    satisfy $\|u-w\|_2 < 1$.

    Now suppose such a vertex does not exist. Let the set $A$ consist
    of all vertices $v \in V(C_k)$ that have both of their cycle neighbours to their right ($v_1 < u_1$ and $v_1 < w_1$). 
    Let the set $B = V(C_k) \setminus A$ consist of
    all vertices $v\in V(C_k)$ that have both of their cycle neighbours to their left ($u_1 < v_1$ and $w_1 < v_1$).
    
    As $k$ is odd, there exist two neighbours $v$ and $w$ on $C_k$ which
    are in the same set: without loss of generality, let $v, w \in B$ and $v_1 < w_1$.
    But then $v \not \in B$, a contradiction.
\end{proof}

\begin{lemma}\label{lem.C5}
    Let $G$ be a distance graph of points $x_1, \dots, x_N$ in $\mathbb{R}^d$,
    and assume each of the points is of distance at most $\frac {\sqrt 3} 4$
    from the line generated by the vector $(1, 0, \dots,0)^T$ (the ``$x$ axis''). 
    Then $G$ does not contain
    an induced cycle of length $5$ or more.
\end{lemma}
\begin{proof}
    Suppose such a cycle $C_k$ exists in $G$. 
    Similarly as in the proof of Lemma~\ref{lem.C5_complement}
    we can assume each vertex in $C_k$ must correspond to a distinct
    point in $\mathbb{R}^d$, since $k\ge 5$ it is easy to obtain a contradiction otherwise.
    
    Let $a$ and $b$ be points on $C_k$ with a minimal and maximal first coordinate respectively.
    Suppose first that $a$ and $b$ are not neighbours on $C_k$.
    Consider the two paths $P_1$ and $P_2$ from $a$ to $b$ that make up $C_k$.
    As $k \ge 5$ the longer path, say $P_2$ has at least two internal vertices.
    Let $u$ be a vertex on $P_1$, $u \not \in \{a,b\}$.
    
    First suppose there exist two neighbours $y$ and $z$ on the path from $P_2$, $y,z \not \in \{a,b\}$ such that $y_1 \le u_1 \le z_1$.
    Since $C_k$ is an induced
    subgraph, $\|u-y\|_2 \ge 1$ and $\|z-u\|_2 \ge 1$. By Lemma~\ref{lem.pyth}
    $\|z-y\|_2 \ge z_1 - y_1 = z_1 - u_1 + u_1 - y_1 \ge 1$,
    a contradiction to the fact that $y$ and $z$ are neighbours on $C_k$ (and $G$). Now if such $y$ and $z$ do not exist, then either all vertices $v \in V(P_2) \setminus \{a,b\}$ satisfy $v_1 \le u_1$, or they all satisfy $v_1 \ge u_1$. Without loss of generality, suppose the former. Let $y$ be the neighbour of $b$ on $P_2$. Since $a_1 \le y_1 \le u_1$ there must be two neighbours $u'$ and $u''$ on the path from $a$ to $u$ on $P_1$ such that $u'_1 \le y_1 \le u_1''$ (it may be that $u'=a$ or $u''=u$). As $P_2$ has at least 2 internal vertices, $y$ cannot be a neighbour of $u'$ or $u''$ on $C_k$. So again $\|u'' - u'\|_2 \ge u_1'' - u_1' \ge u_1'' - y_1 + y_1 - u_1' \ge 1$, a contradiction.

    Now suppose $a$ and $b$ are neighbours on $C_k$. Since $k \ge 5$ there exists
    a vertex $u$ such that $au, bu \not \in E(C_k)$. By the choice of $a$ and $b$,
    $a_1 \le u_1 \le b_1$. Thus again by Lemma~\ref{lem.pyth} $\|a-b\|_2 \ge 1$, a contradiction.
\end{proof}

\begin{corollary}\label{cor.perfect}
    A distance graph $G$ of a sequence of points in $\mathbb{R}^d$ with Euclidean distance at most $\frac {\sqrt 3} 4$
    from a line is a perfect graph.
\end{corollary}
\begin{proof}
    Lemma~\ref{lem.C5_complement}
    and Lemma~\ref{lem.C5} show that neither $G$ nor its complement has an induced
    odd cycle of length 5 or more, in other words, $G$ is a \emph{Berge graph}. The result follows by 
    the strong perfect graph theorem~\cite{strong_perfect}.
\end{proof}

\medskip

By adding $(0, -\frac{\sqrt 3} 4)$ to each of the points $(-1, 0)$, $(0,0)$, $(1,0)$, $(\frac 1 2, \frac {\sqrt 3} 2)$, $(-\frac 1 2, \frac {\sqrt 3} 2)$ and slightly perturbing them it is easy to see that the constant $\frac {\sqrt 3} 4$ is best possible in Lemma~\ref{lem.C5} and Corollary~\ref{cor.perfect}.

In Section~\ref{sec.btk} we extend the corollary to arbitrary normed spaces over the reals. 

\begin{lemma}\label{lem.perfect_gen}
    Let $\mathcal{S}$ be a normed space over the reals. A distance graph $G$ of a sequence of points in $\mathcal{S}$ with distance at most $\frac 1 8$ (at most $\frac {\sqrt 3} 4$ if $\mathcal{S}$ is Hilbert)
    from a line in $\mathcal{S}$ is a perfect graph.
\end{lemma}

\section{Decomposition to uniform measures}
\label{sec.decomposition}

Let $S = (s_1, \dots, s_N)$ be a finite sequence in a metric space. If we ignore the order, $S$ can be viewed as a multiset, so we denote $|S|=N$
and define the uniform distribution $\nu$ on $S$ by $\nu(\{x\}) = |S|^{-1} |\{i \in \{1,\dots,|S|\}: x_i = x\}|$.

\begin{lemma}\label{lem.cliques} 
    Let $\mathcal{S}$ be a metric space with distance $d$. 
    Let $X$ be a random vector distributed uniformly on a finite sequence $S$ in $\mathcal{S}$. 
    Let $G$ be the distance graph of $S$. Then $\concdiam(X,1) = \frac {\omega(G)} {|S|}$ where $\omega(G)$ is the clique number of $G$.
\end{lemma}

\begin{proof}
    By the definition of $\concdiam$ and $G$ it follows that $\concdiam(X, 1) \le \frac {\omega(G)} {|S|}$.
    To prove the opposite inequality, suppose $W$, $W \subseteq V(G)$, forms a clique in $G$, thus $d(x,y) < 1$
    for each $x,y \in W$. We can find a radius $\delta$ small enough that 
    the union of balls of radius $\delta$ with centers at $W$ has diameter at most one.
    Thus $\pr(X \in W) = \frac{|W|} {|S|} \le \concdiam(X,1)$.
\end{proof}

\medskip
For (multi-)sets $A$ and $B$ denote by $A \uplus B$ the operation of multiset union, so the multiplicity of $x$
in $A \uplus B$ is the sum of multiplicities of $x$ in $A$ and $B$ respectively.

\begin{lemma}\label{lem.decomposition}
    Let $\mathcal{S}$ be a normed space over the reals with norm $\|\cdot\|$.
%
    %
    Let $X$ be a random vector distributed uniformly on a finite multiset $S \subset \mathcal{S}$.

    Suppose $\concdiam(X, 1) \le \alpha$ and there is a line $L$ in $\mathcal{S}$ such that
    $S$ is at distance at most $\frac 1 8$ (at most $\frac {\sqrt 3} 4$ in the case $\mathcal{S}$ is Hilbert) from $L$.

    Then 
    there exist 
    $K \le \alpha |S|$ and
    sets $S_1, \dots, S_K$ that partition $S$ ($S=\uplus_{j=1}^K S_j$) such that
    for each $j \in \{1, \dots, K\}$ and each $x,y \in S_j$, $x \ne y$ we have $\|x - y \| \ge 1$.
\end{lemma}

\begin{proof}
    By the strong perfect graph theorem (Corollary~\ref{cor.perfect} or Lemma~\ref{lem.perfect_gen}) and the definition of perfect graphs
    we get that the chromatic number $K$ of the distance graph $G$ on $S$ (with an arbitrary ordering of $S$) is equal to its clique number,
    which by Lemma~\ref{lem.cliques} is at most $\alpha |S|$. 

    Let $S_1, \dots, S_K$ be the colour classes of an optimal colouring. As they are independent
    sets in the distance graph, their vertices are pairwise at distance at least 1.
\end{proof}

\medskip

Recalling that $\nu(S)$ denotes a uniform measure on a (multi-)set $S$, the lemma says that 
$X \sim \nu(S)$ and $\nu(S) = \frac{\sum_{j=1}^K |S_j| \nu(S_j)} {|S|}$, that is $\nu$ is a
mixture of uniform measures on blocks $S_j$.

\section {Arbitrary norms and the B. T. K. chain condition for sets close to a line}
\label{sec.btk}

The main result of this section is Lemma~\ref{lem.btk_strip2} which extends a result of Jones~\cite{jones}
for $d=1$ to general normed spaces and distributions close to a line.

We will need the following concept, which was also used in \cite{LR}.
    Let $\upsilon$ be a unit vector in a normed space $\mathcal{S}$ over the reals.
    We say that $f$ is a \emph{supporting functional} for $\upsilon$ if
    $f$ is a bounded linear functional on $\mathcal{S}$ such
    that $|f(x)| \le \|x\|$ for all $x \in \mathcal{S}$ and
    $f(\upsilon) = 1$.

Note that $f$ as above satisfies $|f(x)|=\|x\|$ for $x \in L_\upsilon$ where $L_\upsilon = \{t \upsilon: t \in \mathbb{R}\}$.
The Hahn--Banach theorem can be used to obtain a supporting functional for any unit vector $\upsilon$ in any normed space $\mathcal{S}$ over the reals: first define $f(x) = t$ if $x = t \upsilon$ for all $x\in L_\upsilon$, then extend $f$ to whole $\mathcal{S}$ while preserving its norm $\|f\|=1$.

In the case when the space is Hilbert, we
can easily generalize Lemma~\ref{lem.pyth} and allow a larger distance from the line. 

\begin{lemma}\label{lem.pyth_gen}
    Let $r$ and $s$ be
    two points in a Hilbert space $\mathcal{S}$ with norm $\|\cdot\|$ and inner product $\langle \cdot, \cdot \rangle$.
    Let $\upsilon$ be a unit vector in $\mathcal{S}$.
    Suppose $\|s-r\| \ge 1$, and both $r$ and $s$ are at distance at most $\frac {\sqrt 3} 4$ from the 
    the line $L_\upsilon$ generated by $\upsilon$. 
    Denote by $r^*=\langle r, \upsilon \rangle \upsilon$ and $s^*=\langle s,\upsilon \rangle \upsilon$ their projections
    onto $L$. Then
    \begin{equation*} 
        |\langle r,\upsilon \rangle - \langle s, \upsilon \rangle| = \|r^*-s^*\| \ge
        \frac 1 2.
    \end{equation*}
\end{lemma}
\begin{proof}
    The proof is analogous to the proof of Lemma~\ref{lem.pyth}.
\end{proof}

\medskip

The next lemma extends this for general normed spaces.

\begin{lemma} \label{lem.supporting_norms}
    Let $\mathcal{S}$ be a normed space over the reals with norm $\|\cdot\|$.
    Let $\upsilon \in \mathcal{S}$ with $\|\upsilon\|=1$, let $L_\upsilon$ be the line generated by $\upsilon$ and let $f$ be 
    a
    supporting functional for $\upsilon$. If $x,y \in \mathcal{S}$ have distance at most $\frac 1 8$ from $L_\upsilon$ and $\|x - y\| \ge 1$ then $|f(x) - f(y)| \ge \frac 1 2$.

    Furthermore, if $\mathcal{S}$ is Hilbert with inner product $\langle \cdot, \cdot \rangle$, we can take $f(x) = \langle x, \upsilon \rangle$ for all $x \in \mathcal{S}$ and the statement still holds for $x,y$ with distance at most $\frac {\sqrt 3} 4$ from $L_\upsilon$.
\end{lemma}

\begin{proof}
    We can write $x = x^* + x'$ and $y=y^* + y'$, where $x^*,y^* \in L$ and $\|x'\|,\|y'\| \le \frac 1 8$. Assume, for example, that $f(x^*) \ge f(y^*)$.
    By the triangle inequality $\|x^* - y^*\| \ge \|x-y\| - \|x' - y'\|$. So
    \begin{align*}
         f(x-y) &= f(x^*-y^*) + f(x' - y') = \|x^* - y^*\| + f(x'-y') 
\\
   & \ge \|x^* - y^*\| - \|f\|\|x' - y'\| = \|x^* - y^*\| - \|x' - y'\|
        \\
        & \ge \|x-y\| - 2\|x'-y'\| \ge 1 - 2(\|x'\|+\|y'\|) \ge \frac 1 2.
    \end{align*}

    The claim with a weaker condition for Hilbert spaces follows by Lemma~\ref{lem.pyth_gen}. Specifically, we do not need the Hahn--Banach theorem in this case. 
\end{proof}

\medskip

$l_1$ norm in the plane can be used to show that $\frac 1 8$ in Lemma~\ref{lem.supporting_norms} is optimal. 
We can now generalize the proof of Corollary~\ref{cor.perfect}.

\medskip

\begin{proofof}{Lemma~\ref{lem.perfect_gen}}
    Let $\upsilon \in \mathcal{S}$ be a unit vector such that all points are at distance at most $\frac 1 8$ (at most $\frac {\sqrt 3} 4$ when $\mathcal{S}$ is Hilbert) from 
    the line $L_{\upsilon}$ generated by $\upsilon$. 
 
    Lemma~\ref{lem.C5_complement} and Lemma~\ref{lem.C5} (with $(1,0,\dots,0)^T$ replaced with $\upsilon$)
    can be generalized to the Hilbert setting using Lemma~\ref{lem.pyth_gen} and to the general setting using Lemma~\ref{lem.supporting_norms} instead of Lemma~\ref{lem.pyth}. In the Hilbert case the projection to the first coordinate $a \mapsto a_1$ can be replaced with the mapping to the inner product $a \mapsto \langle a, \upsilon \rangle$, while in the general case it can be replaced with $a \mapsto f_\upsilon(a)$ where $f_\upsilon$ is the supporting functional for $\upsilon$. Analogous inequalities as those used in the proofs of Lemma~\ref{lem.C5_complement} and Lemma~\ref{lem.C5}
    remain correct by the definition of the supporting functional.

    Finally use the same argument as in Corollary~\ref{cor.perfect}.
\end{proofof}

\medskip

Let $\mathcal{S}$ be a normed space with norm $\|\cdot\|$.
Following \cite{jones, LR}, we say that sets $A, B \subset \mathcal{S}$ \emph{satisfy B.T.K. chain condition} if
the set $A \times B$ can be decomposed into $n$ sets $S_{m-n+1}, S_{m-n+3}, \dots, S_{m+n-1}$,
$m=\max(|A|, |B|)$, $n=\min(|A|, |B|)$
such that $|S_t|=t$ and for each distinct 
$(x, y), (u, v) \in S_t$ we have $\|(x+y) - (u+v)\| \ge 1$, that is, each $S_t$ is a $t$-block.
When we refer to $S \subseteq A \times B$ as a set of \emph{points} we view the pairs $(x,y) \in S$ as elements $x+y \in \mathcal{S}$
like Jones~\cite{jones}. 

\begin{lemma}\label{lem.btk_strip2}
    Let $A=\{x_1, \dots, x_a\}$ and $B=\{y_1, \dots, y_b\}$ be non-empty sets in a normed space $\mathcal{S}$ over the reals with norm $\|\cdot\|$.
    Suppose that for each pair of distinct points $x,y \in A$
    or $x,y \in B$ we have $\|x - y\| \ge 1$. 

    Let $\upsilon$ be a unit vector in $\mathcal{S}$ and let $f$ be a supporting functional for $\upsilon$.
    Assume that $f(x_1) \le \dots \le f(x_a)$, $f(y_1) \le \dots \le f(y_b)$ and $|f(x) - f(y)| \ge \frac 1 2$ for distinct $x,y \in A$  or $x,y \in B$.
   
    Then $A$ and $B$ satisfy the B.T.K. chain condition and furthermore
    for each set $S_t$ in the decomposition and any distinct $x,y \in S_t$
    we have 
    \begin{equation}\label{eq.supporting_separation}
        |f(x) - f(y)| \ge \frac 1 2.
    \end{equation}
\end{lemma}

\begin{proof}
We will adapt the simple idea of \cite{btk}. 
Without loss of generality we can assume that $a \ge b$. 

Define the following matrix
    \[ M = (m_{i,j}) := 
\begin{pmatrix}
x_1 + y_b & \dots & x_a + y_b\\
\dots & \dots & \dots \\
x_1 + y_1 & \dots & x_a + y_1
\end{pmatrix}
\]

Decompose its entries into (ordered) sets
\[
S_{a + b - 2k - 1} := (s_1, \dots, s_{a+b-2k-1}) =  (m_{b - k, 1}, \dots, m_{b-k, a-k}, \dots, m_{1, a-k}),
\]
$k \in \{0, \dots, b - 1\}$. That is, repeatedly peel off the bottommost row (left to right) together with the rightmost column (bottom to top) from the matrix until no entries are left.

The number and the lengths of the sets are as required, so it remains to show that for any indices $i$ and $j$ such that $i < j$ we have $\|s_j - s_i\| \ge 1$ and $f(s_j - s_i) \ge \frac 1 2$.

If $1 \le i < j \le a - k$ we have $\|s_j - s_i\| = \|x_j - x_i\| \ge 1$ and $f(s_j) - f(s_i) \ge \frac 1 2$ by the assumption on $A$ and similarly if $a - k \le i < j \le a+b-2k-1$ we have $\|s_j - s_i\| = \| y_{k + j - (a-k)+1} - y_{k+ i - (a-k)+1}\| \ge 1$ and $f(s_j - s_i) \ge \frac 1 2$ by the assumption on $B$.


Finally, if $1 \le i < a-k < j \le a+b-2k-1$ then using what we have shown
\[
    \|s_j - s_i\| \ge f(s_j - s_i) = f(s_{i+1} - s_i) + \dots + f(s_j - s_{j-1}) \ge \frac {j - i} 2 \ge 1.
\]
\end{proof}

\section{Concentration for sums of random vectors close to a line}
\label{sec.nice_proofs}

The next lemma extends the results of Jones~\cite{jones} in the case when random variables
are close to a line.

\begin{lemma}\label{lem.jones_strip}
    Let $\mathcal{S}$ be a normed space over the reals with norm  $\|\cdot\|$.
    Let $X_1, \dots, X_n$ be independent random vectors,
    such that $X_i$ is distributed uniformly on a $k_i$-block $S_i \subset \mathcal{S}$ (that is, $|S_i| = k_i$ and each pair of distinct points $x,y \in S_i$ satisfies $\|x - y \| \ge 1$).

    Suppose there is a line $L$ in $\mathcal{S}$ such that for each $i \in \{1,\dots,n\}$ each
    point $x \in S_i$ is at distance at most $\frac 1 8$ from $L$ (at most $\frac{\sqrt 3} 4$ if $\mathcal{S}$ is Hilbert).

    Then
    \[
        \concdiam(X_1 + \dots + X_n, 1) \le \pr(Y_1 + \dots + Y_n \in \{0, \frac 1 2\})
    \]
    where $Y_1, \dots, Y_n$ are independent random variables and $Y_i$ is uniformly distributed on $\{1, \dots, k_i\} - \frac{k_i+1} 2$, i.e., in our notation $Y_i \sim \nuopt {k_i^{-1}}$.
\end{lemma}
\begin{proof}
    We use the proof of Theorem~1 of Jones \cite{jones}:
    the B.T.K. chain condition in our case is ensured by Lemma~\ref{lem.btk_strip2}.

    Specifically, we may assume without loss of generality that $L$ contains the origin. Now note that by our assumption and Lemma~\ref{lem.supporting_norms},
    each set $S_i$ is not only a block, but also for each distinct $x,y \in S$ (\ref{eq.supporting_separation})
    holds. Thus for $n \ge 2$, we can apply Lemma~\ref{lem.btk_strip2} to get a decomposition
    of $S_1 \times S_2$ into blocks satisfying the same property. 
    We can then iteratively apply Lemma~\ref{lem.btk_strip2} to each resulting block $A$ and $B=S_3$,
    and so on, until we get a decomposition of $S_1 \times \dots \times S_n$ into blocks.

    As noted by Jones, the number and the sizes of blocks obtained in each step corresponds exactly
    to the number and the sizes of chains obtained in the analogous decomposition, using analogous proof steps, of $\prod_{i=1}^n \{0, \dots, k_i-1\}$ into ``symmetric chains'', see Proposition~1 of \cite{jones} and \cite{btk}.
    The final number of chains (and blocks in our case) $W_{\left \lceil \frac N 2  \right \rceil}$  is shown to be
    equal to the cardinality of the ``middle layer'', which can be defined as
    \[
       \left \{x: \, x \in \prod_{i=1}^n \{0, \dots, k_i-1\} \mbox{ and } \sum_{i=1}^n x_i = \left \lceil \frac N 2  \right \rceil \right \}, \quad N = \sum_{i=1}^n (k_i - 1).
    \]
    Each open ball of diameter at most $1$ contains at most 1 point from a block. Thus the original proof of \cite{jones} yields 
    \[
        \conc(X_1 + \dots + X_n, 1) \le \frac { W_{\left \lceil \frac N 2  \right \rceil}} {\prod_{i=1}^n k_i} = \pr(Y_1' + \dots + Y_n' = \left \lceil \frac N 2 \right \rceil),
    \]
    where $Y_i'$ is uniform on $\{0, \dots, k_i-1\}$ and $\{Y_i'\}$ are independent.
    We have $Y_i \sim Y_i' - \frac {{k_i}-1} 2$,
    so the right side is equal to $\pr(Y_1 + \dots + Y_n \in \{0, \frac 1 2\})$.

    Since each open set of diameter at most $1$ also contains at most one element 
    from a block, 
    we obtain
    a stronger inequality with $\concdiam(\cdot, 1)$ instead of $\conc(\cdot,1)$.
\end{proof}

\medskip


For a sequence $\alpha_1, \dots, \alpha_n \in (0,1]$ define $t(\alpha_1, \dots, \alpha_n):=\pr(Y_1+\dots+Y_n \in \{0, \frac 1 2\})$, where $Y_i \sim \nuopt {\alpha_i}$ and $Y_1, \dots, Y_n$ are independent. 

Recall that $\mu_1 * \mu_2$ denotes the convolution of $\mu_1$ and $\mu_2$ and $*_{i=1}^n \mu_i$ denotes $\mu_1 * \dots * \mu_n$.

\medskip

\begin{proofof}{Theorem~\ref{close_to_line}}
    We first prove ii). 
    Let $c_{\mathcal S}$ be $\frac{\sqrt 3} 4$ if $\mathcal{S}$ is Hilbert and $\frac 1 8$ otherwise. Let $t = \max_i \sup_\omega d(X_i(\omega), L)$. Fix $\epsilon \in (0, \frac {\min_i \alpha_i} 2)$. By Lemma~\ref{lem.discrete_approx_try} (and the construction given in Lemma~\ref{lem.weak_discrete}) 
    there exist $\delta$, $0 < \delta < \frac {c_{\mathcal{S}}} {t} - 1$ and random variables $Z_i$, $i \in \{1, \dots, n\}$ such that 
    \begin{enumerate}
        \item[(a)] $Z_i$ is uniformly distributed on a finite multiset $S_i$ which lies in the support of $(1+\delta)X_i$, so  $\sup_\omega d(Z_i, (1+\delta)L) \le c_{\mathcal{S}}$,
        \item[(b)] $\concdiam(Z_i,1) \le \alpha_i +  \epsilon$ and
        \item[(c)] $\concdiam(Z_1 + \dots + Z_n, 1) \ge \concdiam(X_1 + \dots + X_n, 1) - \epsilon$.
    \end{enumerate}            
    Let $\mu_i \sim \nu(S_i)$ be the distribution of $Z_i$.
    By our decomposition, Lemma~\ref{lem.decomposition}, we have for each $i \in \{1, \dots, n\}$
    \[
        \mu_i = w_{i,1} \mu_{i,1} + ... + w_{i,K_i} \mu_{i,K_i}
    \]
    where $w_{i,j} = \frac {k_{i,j}} {N_i}$, $N_i = |S_i|$, $K_i \le (\alpha_i + \epsilon) N_i$ and 
    $\mu_{i,j}$ 
    is the uniform
    measure on a \emph{set} $S_{i,j}$ of size $k_{i,j}$ with each pair of points at distance at least~1 (a $k_{i,j}$-block).

    We have
    \[
        \mu:=*_{i=1}^n \mu_i = \sum \left(\prod_{i=1}^n w_{i,j_i}\right) *_{i=1}^n \mu_{i, j_i}.
    \]
    where the sum is over all $(j_1, \dots, j_n) \in \{1, \dots, K_1\} \times \dots \times \{1, \dots, K_n\}$.

    Since each $\mu_{i, j_i}$ is uniform on
    a $k_{i,j_i}$-block, 
    by (a) and Lemma~\ref{lem.jones_strip} for any open set $B$ of diameter at most~$1$:
    \[
        (*_{i=1}^n \mu_{i, j_i})(B) \le (*_{i=1}^n \nu^1_{k_{i,j_i}}) (\{0,\frac 1 2\})
    \]
    where $\nu^1_s$ is a uniform measure on $\{1, \dots, s\} - \frac {s+1} 2$. So
    \[
        \mu(B) \le  (*_{i=1}^n \mu_i^1) (\{0, \frac 1 2\})
    \]
    where
    \[
        \mu_i^1 = w_{i,1} \nu^1_{k_{i,1}} + \dots + w_{i,K_i} \nu^1_{k_{i,K_i}}.
    \]
    Now the concentration of $\mu_i^1$ satisfies
    \begin{align*}
        \concdiam(\mu_i^1, 1) \le \sum_{j=1}^{K_i} w_{i,j} \frac 1 {k_{i,j}}
        = \sum_{j=1}^{K_i} \frac {k_{i,j}} {N_i} \frac 1 {k_{i,j}} = \frac {K_i} {N_i} \le \alpha_i+\epsilon.
    \end{align*}
    We have reduced the problem to the 
    case $\mathcal{S} = \mathbb{R}^1$.
    By Theorem~2.1.3 of Juškevičius~\cite{tj} for any $n \ge 1$:
    \[
        \concdiam(X_1 + \dots + X_n, 1) \le \concdiam(*_{i=1}^n \mu_i^1, 1) \le t(\alpha_1+\epsilon, \dots, \alpha_n+\epsilon).
    \]
    Since $t$ is continuous and our inequality holds for any $\epsilon \in (0, \frac {\min_i \alpha_i} 2)$, the claim ii) follows.

    Finally, to prove the theorem in the case i), notice that in this case
    the support of each random variable as well as the support of their sum is
    contained in a countable set.
    The closure of a countable set is separable. So i) follows from ii).
\end{proofof}

\section{A corollary of a result of Hal{\'a}sz}
\label{sec.halasz}

In this section we temporarily switch to the convention of Hal{\'a}sz \cite{halasz} and consider 
concentration for diameter 2 rather than 1.
Results in this section concern only $\mathbb{R}^d$ with the Euclidean norm $\|\cdot\|_2$ and distance $d_2(x,y) = \| x-y\|_2$.

\begin{theorem}{(Hal{\'a}sz (1977), Theorem~4 in~\cite{halasz})}\label{thm.halasz}
    Let $X_k'$ be independent of $X_k$ but of the same distribution, $X_k^* = X_k - X_k'$ and denote by $F_k(x)$ the distribution function of $X^*_k$. Let 
    \[
        F(x) = \sum_{k=1}^n F_k(x).
    \]
    Introducing the notation
    \[
        a_* = \min(a,1) \quad (a \ge 0)
    \]
    let
    \[
        D=\inf_{|e|=1} \int_{\mathbb{R}^d} |\langle x, e \rangle|_*^2 dF(x)
    \]
    and 
    \[
        \mu = \sup_{y \in \mathbb{R}^d} \sum_{k=1}^n \pr (\|X^*_k - y\|_2<1).
    \]
    We have 
    \[
        \conc(\sum_{k=1}^n X_k, 2) \le c_{H,d}\, \mu n^{-1} D^{-d/2} \quad (n \ge 8).
    \]
    The constant $c_{H,d}$ depends on $d$ only.
\end{theorem}

\begin{lemma}\label{lem.halasz_corollary2} 
    Let $d \ge 2$ be an integer. Let $\xi_2(x) = x$ and let $\xi_d(x) = 1$ for $d > 2$.
    
    Let $c \in (0,1)$, $\delta > 0$, $\epsilon > 0$ and $n \ge 8$.
    If $X_1, \dots, X_n$ are independent random vectors in $\mathbb{R}^d$ with $\conc(X_i, 2) \le \alpha_i$, $\alpha_i \in (0,1]$, and $\frac {\sum \alpha_i} n = \bar{\alpha}$ then
    either $\concdiam(\sum_{i=1}^n X_i, 2) < \delta$
    or there exists a line $L$ in $\mathbb{R}^d$ such that $L$ contains the origin and for
    all but at most 
    \[
        9 (d-1)^3 c_H c^{-2}  \delta^{-1}  \epsilon^{-1} \xi_d(\bar{\alpha})
    \]
    indices $i \in \{1, \dots, n\}$ we have $\pr(d_2(X_i, L + a_i) \ge c) \le \epsilon$ with some $a_i \in \mathbb{R}^d$. Here $c_H=c_{H,2}$ from Theorem~\ref{thm.halasz}.
\end{lemma}

Hal{\'a}sz \cite{halasz} notes after his theorem: \textit{``$D$ measures to what extent our distributions are $d$-dimensional. $d$-dimensionality can be measured even in $R^{d_1}$, $d_1 > d$: $|e|=1$ should be replaced by subspaces of dimension $d_1 - d+1$ and $\langle x,e\rangle$ by projection of $x$ on the subspace.''}. This remark implies that Lemma~\ref{lem.halasz_corollary2} should be true with $\bar{\alpha}$ in place of $\xi_d$ for all $d$ and some other $d$-dependent constant. The proof of the remark is not given in \cite{halasz}, and it seems to require a certain technical effort to produce that proof, hence we state our results with $\xi_d(\bar\alpha)$ rather than $\bar{\alpha}$, which is likely suboptimal for $d \ge 3$.

\medskip

\begin{proofof}{Lemma~\ref{lem.halasz_corollary2}}
    Assume $n \ge 8$. Let us first prove the claim for $d=2$. A disk of radius $2$ can be easily covered
    by 9 open disks of radius 1. 
    To see this, center each of the smaller disks on a point in the 3x3 grid with step size, for example, $\sqrt{2} - 0.01$.
    So $\concdiam(\cdot, 2) \le \conc(\cdot, 4) \le 9\conc(\cdot, 2)$.
    Using Theorem~\ref{thm.halasz} and the notation given in its statement 
    \begin{equation}\label{eq.halasz_eq}
        \conc(\sum_{i=1}^n X_i, 2) \le c_H \bar{\alpha} D^{-1}.
    \end{equation}
    Write
    \[
        D_i= \E |\langle X^*_i, e \rangle|_*^2
    \]
    for the unit vector $e$ that achieves the minimum in the definition of $D$, so that $D = \sum_{i=1}^n D_i$.
    If $D > 9 c_H \delta^{-1} \bar{\alpha}$
    then $\conc(\sum_{i=1}^n X_i, 2) < \delta / 9$ and the claim follows.

    Otherwise there can be at most 
    $9 c_H c^{-2}  \delta^{-1} \epsilon^{-1} \bar{\alpha}$
    indices $i$ such that 
    \begin{equation} \label{eq.Di}
        D_i \le c^{2} \epsilon
    \end{equation}
    does not hold.

    Let $L$ be the line containing the origin and orthogonal to $e$. 
    Recall that $X_i^* \sim X_i - X_i'$ where $X_i'$ is an independent copy of $X_i$.
    By conditioning and the fact that $|x|_*^2 \ge \min(c^2, 1) \mathbb{I}_{|x|\ge c} =  c^2 \mathbb{I}_{|x|\ge c}$
    \begin{align*}
        &D_i = \E | \langle X^*_i, e \rangle |_*^2
        \\ & \ge \inf_{a' \in \mathbb{R}^2} \E | \langle X_i - a', e \rangle|_*^2
        \\ & = \E d_2(X_i, L + a_i)_*^2
        \\ & \ge \pr(d_2(X_i, L + a_i) \ge c) c^2
    \end{align*}
    for a vector $a_i \in \mathbb{R}^d$ that realises the infimum. 
    Thus for any $i$ such that (\ref{eq.Di}) holds we have
    \[
        \pr(d_2(Y_i, L + a_i) \ge c) \le c^{-2} D_i \le \epsilon,
    \]
    which completes the proof for the case $d=2$.

    Let us now prove the claim for $d > 2$. If $\concdiam(X_1+\dots+X_n,2) < \delta$ we are done.
    Otherwise, choose any plane $P_1 \subset \mathbb{R}^d$ containing the origin and let $\pi_1$ be 
    the projection of $\mathbb{R}^d$ to $P_1$.
    For a random vector $X$ in $\mathbb{R}^d$ we have $\concdiam(\pi_1 X, 2) \ge \concdiam(X, 2)$.
    Thus $\concdiam(\pi_1(X_1 + \dots + X_n), 2) \ge \delta$. For the summands
    we only have the trivial upper bound $\concdiam(\pi_1 X_i, 2) \le 1$. 

    Let $m=9 (d-1)^2 c_H c^{-2}  \delta^{-1}  \epsilon^{-1}$.
    By the already proved case $d=2$ for $j=1$ there is 
    a unit vector $\tilde{e}_j \in P_j$ 
    and vectors $a_{j,1}, \dots, a_{j,n} \in P_j$ such that
    such that for all but at most $m$
    indices $i \in \{1, \dots, n\}$
    \begin{equation}\label{eq.projection_plane}
        \pr(|\langle X_i - a_{j,i}, \tilde{e}_j \rangle| \ge \frac c {\sqrt{d-1}}) \le \frac {\epsilon} {d-1}.
    \end{equation}
    Let us do similarly for each $j=2, \dots, d-1$: at step $j$ let 
    $\hat{e}_{j-1}$ be a unit vector in $P_{j-1}$ orthogonal to $\tilde{e}_{j-1}$. Since $j \le d-1$
    we can pick a unit vector $e_j \in \mathbb{R}^d$ orthogonal to $\tilde{e}_1, \dots, \tilde{e}_{j-1}$ and $\hat{e}_{j-1}$.
    Consider a plane $P_j \subset \mathbb{R}^d$ generated 
    by $e_j$ and $\hat{e}_{j-1}$. Again there is a unit vector $\tilde{e}_j$ and vectors $a_{j,1}, \dots, a_{j,n}$ in $P_j$ such that for all but
    at most $m$ indices $i$ (\ref{eq.projection_plane}) holds.
    Finally, define $\tilde{e}_d$ to be a unit vector orthogonal to $\tilde{e}_1, \dots, \tilde{e}_{d-1}$.

    We can assume that $\tilde{e}_1, \dots, \tilde{e}_d$ are the standard unit vectors $(1, 0,\dots, 0)^T$, $\dots$, $(0, \dots, 0, 1)^T$. Define $a_i = ((a_{1,i})_1, \dots, (a_{d-1, i})_{d-1}, 0)^T$. Now in (\ref{eq.projection_plane}) we can replace $a_{j,i}$ by $a_i$ and in the new coordinate  system it is equivalent to
    \[
        \pr(|(X_i - a_i)_j| \ge \frac {c} {\sqrt{d-1}}| \le \frac {\epsilon}{d-1}.
    \]
    It follows by the union bound that if $L$ is the line generated by $\tilde{e}_d$, then for all but at most $(d-1) m$ indices
    \[
        \pr(d_2(X_i - a_i, L) \ge c) \le \epsilon.
    \]
\end{proofof}

\section{The general result and its proof}
\label{sec.general}

When $\alpha^{-1}$ is integer $\nu^*_\alpha$ is uniform and its variance is $f(\alpha)$, where $f(x)=\frac 1 {12} (x^{-2} - 1)$. From (\ref{eq.nu_star}) we get $\alpha = \frac 1 {k+1} + \frac p {k(k+1)}$, where $k=\lfloor \alpha^{-1} \rfloor$. So for any $\alpha \in (0, 1]$ the variance $V_\alpha^*$ of $\nu^*_{\alpha}$ is a piecewise linear function of $\alpha$ which agrees with $f$ at points $k^{-1}$, $k\in\{1,2,\dots\}$
and is its linear interpolation between subsequent such points. It can also be written as
\begin{equation}\label{eq.V_alpha}
    V_\alpha^*  =  \frac 1 {12} \lfloor \alpha^{-1} \rfloor  (1 + \lfloor \alpha^{-1} \rfloor) (3 - \alpha - 2 \alpha \lfloor \alpha^{-1} \rfloor).
\end{equation}
%

In order to complete our proofs we need precise asymptotics of the concentration of sums of independent variables with distributions $\nuopt{\alpha_i}$. A standard method for this is the local limit theorem. However, the existing variants we are aware of are difficult to apply when some $\alpha_i^{-1}$ are close to an integer and some are not. Therefore we prove our own lemma using a central limit theorem (Berry--Esseen) and the properties of $\nuopt{\alpha_i}$.

Write $x = y \pm \epsilon$ as a shorthand for $x \in [y-\epsilon, y+\epsilon]$.
\begin{lemma} \label{lem.clt}
    Let $1 \ge \alpha_1 \ge \dots \ge \alpha_n > 0$.
    Let $Y_1, \dots, Y_n$ be independent random variables with distributions $\nu^*_{\alpha_1}, \dots, \nu^*_{\alpha_n}$ respectively. Let
    \[
        V_t^* := \E Y_1^2 + \dots + \E Y_t^2 = V_{\alpha_1}^* + \dots + V_{\alpha_t}^*; \quad V^* := V_n^* > 0.
    \]
    If there is $c \in (0,1)$ such that
    \begin{equation}\label{eq.cVstar}
        V_{\lceil n(1-c) \rceil}^* \ge \frac 1 2 V^*
    \end{equation}
    and $\delta' \in (0,1)$ such that
    \begin{equation}\label{eq.berryesseen}
        \E |Y_1|^3 + \dots + \E |Y_n|^3 \le \delta' (V^*)^{3/2}
    \end{equation}
    then assuming
    \begin{equation}\label{eq.clt_eps}
        \epsilon' := 405 (\delta')^{\frac 1 2} c^{-\frac 3 4} \le \frac 1 2.
    \end{equation}
    we have
    \begin{align*}
        t(\alpha_1, \dots, \alpha_n) = \pr(Y_1 + \dots + Y_n \in \{0, \frac 1 2\}) = \frac {1 \pm \epsilon'} {\sqrt {2 \pi V^*}}.
    \end{align*}
\end{lemma}
Since $V_{\alpha}^*$ is convex, Jensen's inequality implies that for a fixed $\bar{\alpha}$ the main asymptotic term $\frac 1 {\sqrt{2 \pi V^*}}$
is maximized when $\alpha_1=\dots= \alpha_n = \bar{\alpha}$: 
\begin{remark}\label{rmk.explicit_bound}
    Let $Y_1, \dots, Y_n$, $V^*$ and $\bar{\alpha}$ be as in Lemma~\ref{lem.clt}. Then
    \[
        \frac 1 {\sqrt {2 \pi V^*}} \le  \frac 1 {\sqrt {2 \pi n V_{\bar{\alpha}}}} \le \sqrt {\frac 6 {\pi n}}  \frac {\bar{\alpha}} {\sqrt{ 1 - \bar{\alpha}^2}}.
    \]
\end{remark}

\begin{proof} 
    Recall that $V^*_\alpha$ is given in (\ref{eq.V_alpha}) and it is a linear interpolation of $f$, $f(\alpha) = {12}^{-1}(\alpha^{-2}-1)$.
    $V^*_\alpha$ is a convex function of $\alpha$ and $V^*_\alpha \ge f(\alpha)$ for $\alpha \in (0,1]$. The claim follows since by Jensen's inequality
    \[
        V^* = \sum_{i=1}^n V^*_{\alpha_i}  \ge n V^*_{\bar{\alpha}} \ge n f(\alpha).
    \]
\end{proof}

The next theorem is the most general and most detail one in this paper. It combines all of the results above.

\begin{theorem} \label{thm.main}
    For each integer $d \ge 2$ and each norm $\|\cdot\|$ on $\mathbb{R}^d$ there is a constant $C = C(d,\|\cdot\|)$ such that the following holds.

    Let $X_1, \dots, X_n$ be independent random vectors in $(\mathbb{R}^d, \|\cdot\|)$; $n \ge 8$.
    Suppose $\concdiam(X_i, 1) \le \alpha_i$ where $1 \ge \alpha_1 \ge \dots \ge \alpha_n > 0$. 
    Let $Y_1, \dots, Y_n$ be independent random variables with the extremal distributions $\nu^*_{\alpha_1}, \dots, \nu^*_{\alpha_n}$ respectively.
    Write  
    \[
        \bar{\alpha} = {n}^{-1} \sum_{i=1}^{n} \alpha_i.
    \]
    Let $V^*_t, V^*$ be as in Lemma~\ref{lem.clt}. Suppose there is $c \in (0, \frac 1 3)$ such that
    \begin{equation}\label{eq.cVstar2}
        V^*_{\lceil(1-c) n\rceil} \ge \frac 3 4 V^*
    \end{equation}
    and conditions (\ref{eq.berryesseen}) and (\ref{eq.clt_eps})
    hold for some positive $\delta'$ and $\epsilon'$ as defined in Lemma~\ref{lem.clt}.
    Suppose further that 
    $\epsilon' \le \frac 3 {16}$, 
    \begin{equation}\label{eq.condition_near_one}
        \xi_d(\bar{\alpha}) \bar{\alpha}^2 n \le (V^*)^{3/2}\gamma \mbox{ for some }\gamma \le  (10 C)^{-2} 
    \end{equation}
    where $\xi_2(x)=x$ and $\xi_d(x)=1$ for $d>2$,
    and    
    \begin{equation}\label{eq.m_eps}
        m:= C \xi_d(\bar{\alpha})^{\frac 1 2} t(\alpha_1, \dots, \alpha_n)^{- \frac 1 2} n^{\frac 1 2} < \frac {c n } 5.
    \end{equation}
    Then
    \begin{align*}\label{eq.main_inequality}
        \concdiam(X_1 + \dots + X_n, 1) 
          \le \frac {1 + 6 \epsilon' + 4 m n^{-1} + C \sqrt{\gamma}} {\sqrt {2 \pi V^*_{n-\lfloor m \rfloor}}} + e^{-\frac{m^2} {9 n}}.
    \end{align*}
\end{theorem}

Let us prove Lemma~\ref{lem.clt} first.

\medskip

\begin{proofof}{Lemma~\ref{lem.clt}}
    Let $S = Y_1 + \dots + Y_n$.

    Write $k_i = \lfloor \alpha_i^{-1} \rfloor$. Then $p_i$ (corresponding to $p$ in (\ref{eq.nu_star}))
    is $p_i = k_i (\alpha_i (k_i+1) - 1)$.

    Let $A = \{i: Y_i \in \frac 1 2 + \mathbb{Z}\}$. For each fixed subset of indices $A'$ such that $\pr(A = A') > 0$ let $S_{A'}$ be distributed as $S$ given the event $A = A'$.
    Note that $S_{A'}$ has either support on integers (when $|A'|$ is even) or on half-integers (when $|A'|$ is odd). By \cite{keilsongerber} it is unimodal on its support and log-concave: $\pr(S_{A'} = x)^2 \ge \pr(S_{A'}=x-1) \pr(S_{A'}=x+1)$ for all $x$.

    Consequently, $S$ conditioned on the event that $|A|$ is odd (even) is symmetric, unimodal and distributed on half-integers (integers).
    So $S$ itself is symmetric and unimodal when restricted to half-integers or to integers. Therefore for $a \in \mathbb{Z} \cup (\frac 1 2 + \mathbb{Z})$ and $t \in \mathbb{Z}$, with $a \ge 0$, $t >0$
\[
    \pr(S \in \{a, a+\frac 1 2\}) \ge \frac {\pr(S \in [a, a+t-\frac 1 2])} {t} \ge \pr(S \in \{a+t-1, a+t-\frac 1 2\}).
\]
    We will next see that the above monotonicity argument together with a central limit theorem easily yields a useful lower bound. For the upper bound we will additionally need to condition on $A$, slightly more (but standard) calculations and, crucially, the log-concavity property of $S_{A'}$.

    Let us continue the proof. The assumption (\ref{eq.berryesseen}) allows us to apply the Berry--Esseen theorem. This theorem is known to hold with a constant smaller than 1, see, e.g., \cite{shevtsova}, that is, for any $z \in \mathbb{R}$ we have
\[
    |\pr(\frac{S}{\sqrt{V_{n}^*}} \le z) - \Phi(z)| \le \delta' 
\]
so
\[
    \pr(S \in (x \sqrt {V^*}, y \sqrt{V^*}]) = \Phi(y) - \Phi(x) \pm 2 \delta'.
\]
    Therefore for any $\tilde{\epsilon} > 0$
\begin{align*}
    &\pr(S \in \{0,\frac 1 2\}) \ge \frac {\pr(S \in [0, \tilde{\epsilon} \sqrt{V^*}])} {\tilde{\epsilon} \sqrt {V^*}+1}
    \\  & \ge \frac 1 {\tilde{\epsilon} \sqrt{V^*}+1} (\Phi(\tilde{\epsilon}) - \Phi(0) - 2 \delta').
\end{align*}
    By Taylor's theorem for any real $x$ and $h > 0$ we have  $\Phi(x+h) = \Phi(x) + h \Phi'(x) + \frac{\Phi''(\xi)} 2 h^2$ for
    some $\xi \in [x, x+h]$. Since $\Phi'(0)=\frac 1 {\sqrt {2 \pi}}$, $|\Phi''(x)| \le \frac 1 {\sqrt {2 e \pi}}$ for all $x \in \mathbb{R}$
    and $(1 + h)^{-1} \ge 1 - h$ for $h > 0$ and $(1-h_1) (1-h_2) \ge 1-h_1-h_2$ for any $h_1, h_2 \ge 0$
    \begin{align}\label{eq.extremal_local_lower}
        &\pr(S \in \{0,\frac 1 2\}) \ge \frac 1 {\sqrt {V^*} + \tilde{\epsilon}^{-1}} \left(\frac 1 {\sqrt {2 \pi}}  - \frac {\tilde{\epsilon}} {2 \sqrt {2 e \pi}} - 2 \delta' \tilde{\epsilon}^{-1} \right) \nonumber
        \\ &
         \ge \frac 1 {\sqrt {2 \pi V^*}} \left(1  - \frac {\tilde{\epsilon}} {2 \sqrt {e}} - 2 \delta' \tilde{\epsilon}^{-1} {\sqrt {2 \pi}} - \frac 1 {\tilde{\epsilon} \sqrt{V^*}} \right )
        .
\end{align}

    Deferring the choice of $\tilde{\epsilon}$ to the end of the proof, we now turn to the upper bound.

    For any $\alpha \in (0,1]$ let $Y_{\alpha}$ be a random variable with distribution $\nuopt{\alpha}$ and recall that $V_{\alpha}^* = \E Y_{\alpha}^2$.
    Clearly if 
    $k=\lfloor \alpha^{-1} \rfloor$ then
    ${V_{k^{-1}}^*} \le V_{\alpha}^* \le {V_{(k+1)^{-1}}^*}$.
    Furthermore, if $k \ge 2$ we have
    \begin{equation}\label{eq.bound_var}
        1 \le \frac {V_{(k+1)^{-1}}^*}{V_{k^{-1}}^*} = \frac { (k+1)^2 - 1} {k^2 - 1} \le \frac 8 3.
    \end{equation}
    It is also easy to see that $\E |Y_{k^{-1}}|^3 \le \E |Y_{\alpha}|^3 \le \E |Y_{(k+1)^{-1}}|^3$.
    We claim that for $k \ge 2$
    \begin{equation}\label{eq.bound_thirdm}
         \E |Y_{(k+1)^{-1}}|^3 \le 8 \E |Y_{k^{-1}}|^3.
    \end{equation}
    To see this, note that if $k$ is even 
    then  $|Y_{k^{-1}}| + \frac 1 2$ stochastically dominates $|Y_{(k+1)^{-1}}|$ and $|Y_{k^{-1}}| \ge \frac 1 2$,
    so (\ref{eq.bound_thirdm}) follows. If $k$ is odd (so $k \ge 3$) we may obtain $\nuopt {(k+1)^{-1}}$ from $\nuopt {k^{-1}}$ by shifting the mass of each of the $k-1$ non-zero atoms $x$ of $\nuopt {k^{-1}}$ to the atoms $\{x-\frac 1 2, x+\frac 1 2\}$, so that the mass at each of the new atoms is $\frac {k-1} {k(k+1)}$
    and multiplying the measure by $\frac k {k-1}$. This implies that $\frac {  \E |Y_{(k+1)^{-1}}|^3 } {  \E |Y_{k^{-1}}|^3  } \le \max_{x \ge 1} \frac {k {(x + \frac 1 2)^3}} {x^3(k-1)}
    \le \frac 3 2 (\frac 3 2)^3<8$. 

    Let $t$ be the largest index such that $k_1= \dots = k_t = 1$. 
    Define $V_{t,A}^* := \sum_{i=1}^t \E (Y_i^2 | A)$. 
    $4 V_{t, A}^*$ is equal to $|\{1, \dots, t\} \setminus A|$, which is a sum of $t$ independent Bernoulli random variables with parameters $1-p_1$, $\dots$, $1-p_t$ and 
    $\E V_{t,A}^* = V_{t}^* = \sum_{i=1}^t \E Y_i^2 = \frac 1 4 \sum_{i=1}^t (1-p_i)$. By well known concentration inequalities (e.g. Theorem~2.3 of \cite{cmcd98})
    \begin{equation}\label{eq.conc}
        \pr(|V_{t,A}^* - V_{t}^*|  > 0.5 V_{t}^*) \le 2 e^{-\frac 3 {28} V_{t}^*}.
    \end{equation}
    Also since $k_i=1$ for $i \in \{1, \dots, t\}$
    \begin{equation}\label{eq.bound_thirdmc}
       \E(\sum_{i=1}^t |Y_i|^3 | A) =  \frac 1 2 V_{t,A}^*.
    \end{equation}
    Let $\epsilon$ be small, but not too small:
    \begin{equation}\label{eq.epsilon_not_too_small}
        \epsilon > 2 e^{-\frac {3 V^*} {112}}.
    \end{equation}

    Write $V_{A}^* = \sum_{i=1}^n \E (Y_i^2 | A)$.
    Let $C$ be the event that 
    \begin{align}\label{eq.defC}
        \frac 3 {16} \le \frac {V_{A}^*} {V^*} \le \frac {8} {3 c} \quad\mbox{and}\quad
    \E (\sum_{i=1}^n |Y_i|^3 | A) \le \frac 8 c \E (\sum_{i=1}^n |Y_i|^3).
    \end{align}
    Here the constant $c$ is from (\ref{eq.cVstar}).
    We claim that
    \begin{equation}\label{eq.claim_variance_bound}
        \pr(\bar{C}) \le \frac {\epsilon} {\sqrt{V^*}}.
    \end{equation}
    First suppose that the right side of (\ref{eq.conc}) is at most $\frac {\epsilon} {\sqrt{V^*}}$.
    Then on the event $|V_{t,A}^* - V_{t}^*|  \le 0.5 V_{t}^*$ we have, using (\ref{eq.bound_var})
    \begin{align*}
        V_{A}^*  = V_{t,A}^* + \sum_{i=t+1}^n \E (Y_i^2 | A) \ge \min(\frac 1 2, \frac 3 8) V^* \ge \frac {3 V^*} 8,
    \end{align*}
    and similarly $V_{A}^* \le \frac 8 3 V^*$.

    Similarly bounding the contribution of $i \le t$ using (\ref{eq.bound_thirdmc}) and the contribution of $i > t$ using (\ref{eq.bound_thirdm}), we get that on the event $|V_{t,A}^* - V_{t}^*|  \le 0.5 V_{t}^*$ we have  $\sum_{i=1}^n \E (|Y_i|^3 |A) \le 8 \sum_{i=1}^n \E |Y_i|^3 $, thus this event implies $C$ and (\ref{eq.claim_variance_bound}) follows.

    Now suppose the opposite 
    \[
        2 e^{-\frac 3 {28} V_{t}^*} >  \frac {\epsilon} {\sqrt{V^*}}.
    \]
    We will show that $C$ always holds.
    The last inequality is equivalent to
    \begin{align*}
        & -\frac {3 V_{t}^*} {28} > \ln \frac \epsilon 2 - \frac 1 2 \ln V^*; 
        \\ & V_{t}^* < \frac {28} 3 \ln \frac 2 \epsilon + \frac {14} 3 \ln V^*. 
    \end{align*}
    We can rewrite (\ref{eq.epsilon_not_too_small}) as
        $\frac {28} 3 \ln \frac 2 \epsilon < \frac {V^*} 4$. 
        We can assume $V^* \ge 83$, see the end of the proof. Thus we have $\frac {14} 3 \ln V^*  \le \frac {V^*} 4$, so
        \begin{equation*} 
            \frac {28} 3 \ln \frac 2 \epsilon  + \frac {14} 3 \ln V^* < \frac {V^*} 2.
        \end{equation*} 
        Combining the previous two inequalities we get $V_{t}^* < \frac {V^*} 2$.
        Therefore by the assumption (\ref{eq.cVstar}) of the lemma $t \le \lceil (1-c) n \rceil - 1 \le (1-c) n$.  
        Also by (\ref{eq.bound_var})
    \[
        V_{A}^* \ge \sum_{i=t+1}^n \E (Y_i^2 | A) \ge \frac 3 8 \sum_{i=t+1}^n \E Y_i^2 = \frac 3 8 (V^* - V_{t}^*) \ge \frac 3 {16} V^*.
    \]
    As $V_{\lceil \alpha^{-1} \rceil^{-1}}^* \le V_{\lfloor \beta^{-1} \rfloor^{-1}}^*$
    for $1 \ge \alpha \ge 0.5 \ge \beta > 0$, we get
    \begin{align*}
        &V_{t,A}^* \le (1-c) V_{A}^*;\\
        &V_{t,A}^* \le \frac {1-c} c (V_{A}^* - V_{t,A}^*).
    \end{align*}
    So using (\ref{eq.bound_var})
    \begin{align*}
        &V_{A}^* \le \frac 8 3 V^* (1 + \frac {1-c} c) = \frac {8 V^*} {3c}.
    \end{align*}
    Similarly using (\ref{eq.bound_thirdmc}) and  (\ref{eq.bound_thirdm}) we get
    \[
        \E(|Y_1|^3 + \dots + |Y_n|^3 | A) \le \frac 8 c \E(|Y_1|^3 + \dots + |Y_n|^3),
    \]
    which completes the proof of (\ref{eq.claim_variance_bound}).

    Now fix a set $A' \subseteq \{1, \dots, n\}$ such that the event $A=A'$ implies $C$.
    Let $V_{A'}^* = Var(S_{A'})$. 
    By (\ref{eq.defC}) and (\ref{eq.berryesseen})
    \begin{align*}
        &\mathbb{I}_C\E(|Y_1|^3 + \dots + |Y_n|^3 | A)
        \le \delta' c' \mathbb{I}_C (V_{A}^*)^{3/2},
    \end{align*}
    $c'  = \frac 8 c (\frac {16} 3)^{3/2}$.
    Thus similarly as above by the Berry--Esseen theorem for $z\in \mathbb{R}$
    \[
       \left |\pr\left(\frac{S_{A'}}{\sqrt{V_{A'}^*}} \le z\right) - \Phi(z) \right| \le \delta'c'. 
    \]
    Set $x = \frac 1 2$ if $|A'|$ is odd and $x=0$ otherwise.
    Fix small $\epsilon_1$ and $\epsilon_2$ such that $0 \le \epsilon_1 < \epsilon_2 < 1$
    but 
    \begin{equation}\label{eq.eps2eps1}
        {(\epsilon_2 - \epsilon_1) \sqrt{V^*}} \ge 2. 
    \end{equation}
    Define $I_{\epsilon_1, \epsilon_2} = (x + \mathbb{Z}) \cap (\epsilon_1 \sqrt{V^*}, \epsilon_2 \sqrt{V^*}]$.
    Let $r_{A'} = \frac {V_{A'}^*} {V^*}$.
    We have
    \[
        \pr(S_{A'} \in I_{\epsilon_1, \epsilon_2}) = \Phi(r_{A'}^{-\frac 1 2} \epsilon_2) - \Phi(r_{A'}^{-\frac 1 2} \epsilon_1) \pm 2 \delta'c'
    \]
    Let $m=\min I_{\epsilon_1, \epsilon_2}$
    Using a similar argument as above, since $S_A'$ is unimodal on $x + \mathbb{Z}$, and $-(2e \pi)^{-1/2} \le \Phi^{(2)}(z) \le 0$ for $z \ge 0$,
    using Taylor's theorem twice
    there is $\xi \in [0, r_{A'}^{-\frac 1 2}\epsilon_2]$ such that
    \begin{align*}
        &\pr(S_{A'} = m) \ge \frac {\Phi(r_{A'}^{-\frac 1 2} \epsilon_2) - \Phi( r_{A'}^{-\frac 1 2} \epsilon_1) - 2 \delta'c'} {|I_{\epsilon_1, \epsilon_2}|} \nonumber
        \\ & \ge \frac {\Phi( r_{A'}^{-\frac 1 2} \epsilon_2) - \Phi( r_{A'}^{-\frac 1 2}  \epsilon_1) - 2 \delta'c'} {(\epsilon_2 - \epsilon_1) \sqrt{V^*} + 1} \nonumber
        \\ & \ge  \frac {\Phi'(0)r_{A'}^{-\frac 1 2}(\epsilon_2 - \epsilon_1) + 2^{-1}\Phi^{(2)}(\xi) r_{A'}^{-1} (\epsilon_2^2 - \epsilon_1^2) - 2 \delta'c'} {(\epsilon_2 - \epsilon_1) \sqrt{V^*} + 1} \nonumber
        \\ & \ge \frac 1 {\sqrt {2 \pi V^*}} (r_{A'}^{-\frac 1 2} - \frac{r_{A'}^{-1} (\epsilon_2^2 - \epsilon_1^2)} {2 \sqrt e (\epsilon_2-\epsilon_1)}  - \frac {2 \sqrt{2 \pi} \delta'c'} {\epsilon_2 - \epsilon_1}) (1 - \frac 1 {(\epsilon_2 - \epsilon_1) \sqrt{V^*}}) \nonumber
    \end{align*}
        \begin{align} 
            \pr(S_{A'} = m) \ge \frac {1} {\sqrt {2 \pi r_{A'} V^*}} (1 - \frac {\epsilon_1 + \epsilon_2} {2 \sqrt {e r_{A'}}} - \frac {2 \sqrt{2 \pi r_{A'}} \delta'c'} {(\epsilon_2 - \epsilon_1)} - \frac 1 {(\epsilon_2 - \epsilon_1) \sqrt{V^*}}) 
        . \label{eq.clt_bound1}
    \end{align}
    Similarly for $M=\max I_{\epsilon_1, \epsilon_2}$, since  
    \begin{equation}\label{eq.z1z2}
        \frac {1+z_1} {1-z_2} \le 1 + 2 z_1 + 2z_2 \quad\mbox{for }  z_2 \in (0, \frac 1 2) \mbox{ and }z_1 > 0,
    \end{equation}
    using (\ref{eq.eps2eps1}) we get 
    \begin{align} \label{eq.clt_bound2}
        &\pr(S_{A'} = M) \le \frac {1} {\sqrt { 2 \pi r_{A'} V^*}} (1  
        + \frac {4 \sqrt {2 \pi r_{A'}} \delta'c'}{(\epsilon_2 - \epsilon_1)}  + \frac 2 {(\epsilon_2 - \epsilon_1) \sqrt{V^*}}).
    \end{align}
    Fix $\tilde{\epsilon} >  2 (V^*)^{-1/2}$.
    Denote $\tau=\lfloor \tilde{\epsilon} \sqrt{V^*} \rfloor$.
    Applying (\ref{eq.clt_bound1}) with $\epsilon_1 = 2\tilde{\epsilon}$ and $\epsilon_2= 3\tilde{\epsilon}$, then using (\ref{eq.defC}) we get 
    \begin{align*}
        &
        \pr(S_{A'} = x + 2 \tau) 
        \ge  \frac {1} {\sqrt {2 \pi r_{A'} V^*}} (1 - h_1); 
        \\ & 
        h_1 =   \frac {10\tilde{\epsilon}} {\sqrt{3 e}} + 8 \sqrt{\frac \pi {3c} } \delta'c' \tilde{\epsilon}^{-1} + (\tilde{\epsilon} \sqrt{V^*})^{-1}.
    \end{align*}

    Similarly applying (\ref{eq.clt_bound2}) with $\epsilon_1 = 0$ and $\epsilon_2=\tilde{\epsilon}$, then using (\ref{eq.defC}) we get 
    \begin{align*}
        &
        \pr(S_{A'} = x + \tau)  
         \le \frac {1} {\sqrt { 2 \pi r_{A'} V^*}} (1 + h_2);
        \\ &
        h_2  = 16 \sqrt{\frac \pi {3c} }\delta'c' \tilde{\epsilon}^{-1}  + 2 (\tilde{\epsilon}\sqrt{V^*})^{-1}.
    \end{align*}
Now since $S_{A'}$ is unimodal and log-concave a `telescoping' argument implies that
%





    \begin{equation}\label{eq.logconct}
        \pr(S_{A'} = x) \pr(S_{A'} = x + 2 \tau) \le \pr(S_{A'} = x + \tau)^2.
    \end{equation}
    So
    \[
        \pr(S_{A'} = x) (1-h_1) \le \pr(S_{A'} = x + \tau) (1+h_2).
    \]
    For all $A'$ (even or odd size)
    this can be written as
    \[
        \pr(S_{A'} \in \{0, \frac 1 2\}) (1-h_1) \le \pr(S_{A'} \in \{\tau, \tau + \frac 1 2\}) (1 + h_2).
    \]
    Averaging over all $A'$ such that $C$ holds on $A=A'$ this becomes
    \[
        \pr(S \in \{0, \frac 1 2\} \cap C) (1-h_1) \le \pr(S \in \{\tau, \tau + \frac 1 2\} \cap C) (1 + h_2).
    \]
    Assume 
\begin{equation}\label{eq.epsilons_assumption}
\pr(\bar{C}) + h_1 + h_2 < \frac 1 4.
\end{equation}
    For any event $A$, $\pr(A) \le \pr(A \cap C) + \pr(\bar{C})$ 
    so using (\ref{eq.z1z2}) and (\ref{eq.epsilons_assumption}) 
    \begin{align}\label{eq.S0tau}
        &  \pr(S \in \{0, \frac 1 2\}) \le  \pr(S \in \{\tau, \tau + \frac 1 2\}) ( 1 + 2h_1 + 2h_2) + \pr(\bar{C}).
    \end{align}
    By a simple calculation:
    \[
        2h_1 + 2 h_2 =\frac {20} {\sqrt {3 e} } \tilde{\epsilon} + 48 \sqrt{\frac {\pi} {3 c}} \frac 8 c 
        \left(\frac {16} {3}\right)^{\frac 3 2} 
        \frac {\delta'} {\tilde{\epsilon}} + \frac 6 {\tilde{\epsilon} \sqrt{V^*}} \le 8 \tilde{\epsilon} + 
        \frac { 4840
        \delta'} { c^{\frac 3 2} \tilde{\epsilon} } + \frac {6} {\tilde{\epsilon} \sqrt{V^*}}.
    \]
    Similarly as in (\ref{eq.extremal_local_lower}) and (\ref{eq.clt_bound2}) we get (since $\tilde{\epsilon} \sqrt{V^*} > 2$)
    \begin{equation}\label{eq.Stau}
         \pr(S \in \{\tau, \tau+\frac 1 2\}) \le
        \frac 1 {\sqrt {2 \pi V^*}} \left(1  + {4 \delta'} {\tilde{\epsilon}}^{-1} \sqrt {2 \pi}  + \frac 2 {\tilde{\epsilon} \sqrt{V^*}} \right )
    \end{equation}
    Combining (\ref{eq.S0tau}) and (\ref{eq.Stau}) and using (\ref{eq.extremal_local_lower}), (\ref{eq.claim_variance_bound}) and (\ref{eq.epsilons_assumption}) yields
    \begin{align*}
        & \pr(S \in \{0, \frac 1 2\}) 
        \\ &
         \le \frac 1 {\sqrt {2 \pi V^*}} \left(1  + {4 \delta'} {\tilde{\epsilon}}^{-1} \sqrt {2 \pi} + \frac 2 {\tilde{\epsilon} \sqrt{V^*}} \right )
         ( 1 +  2h_1 + 2h_2) + \pr(\bar{C})
        \\ &
\le
         \frac 1 {\sqrt {2 \pi V^*}} (1 +  2h_1 + 2h_2 +   {6 \delta'} {\tilde{\epsilon}}^{-1} \sqrt {2 \pi} + \frac 3 {\tilde{\epsilon} \sqrt{V^*}} ) + \frac {\epsilon} {\sqrt{V^*}}.
        \\ &
\le
        \frac 1 {\sqrt {2 \pi V^*}} (1 + h);
        \\ &
        h=   
        8 \tilde{\epsilon} +
        \frac { 4840 \delta'} { c^{\frac 3 2} \tilde{\epsilon} } 
        + \frac {9} {\tilde{\epsilon} \sqrt{V^*}}
        +   {6 \delta'} {\tilde{\epsilon}}^{-1} \sqrt {2 \pi} +  \sqrt{2 \pi} \epsilon. 
    \end{align*}
    We have
    \[
        h \le \frac 1 {\sqrt{V^*}} \frac {9} {\tilde \epsilon} + 8 \tilde{\epsilon} + \frac {4856 \delta'} {c^{\frac 3 2} \tilde{\epsilon}}  + \sqrt{2 \pi} \epsilon.
    \]
    Put $\epsilon = \frac 2 {\sqrt{V^*}}$. Since $V^* \ge 83$, (\ref{eq.epsilon_not_too_small}) is satisfied. The minimum $x_0$ of $f(x) = \frac a x + b x$ with $a,b>0$ is $x_0 = \sqrt \frac a b$ with $f(x_0)=2\sqrt {a b}$. Hence if we put $\tilde{\epsilon} = {8}^{-1/2}(\frac {9} {\sqrt{V^*}} + \frac {4856 \delta'} {c^{3/2}})^{1/2}$, we get
\begin{align}
    h & \le 2 \times \sqrt{8 \left( \frac 9 {\sqrt{V^*}} + 4856 \delta' c^{- \frac 3 2} \right)} + \frac {2 \sqrt {2 \pi}} {\sqrt{V^*}} \nonumber
    \\ & \le 2 \times \sqrt{8 \left( 18 \delta' + 4856 \delta' c^{- \frac 3 2} \right)} + 4 \sqrt {2 \pi} \delta' \nonumber
    \\ & \le 2 \times \sqrt{8 \cdot 4874 \delta' c^{- \frac 3 2}} + {4 \sqrt {2 \pi}} (\delta')^{\frac 1 2} \nonumber
    \\ & \le 405 (\delta')^{\frac 1 2} c^{-\frac 3 4} = \epsilon'. \label{eq.h_epsilon}
\end{align}
    Here in the second line we used 
    \begin{equation}
        \delta' \ge \frac 1 {2 \sqrt{V^*}}, \label{eq.deltaV}
    \end{equation}
    which follows since by (\ref{eq.berryesseen}) and the definition of $\nuopt \alpha$
    \[
        \frac 1 2 V^* = \frac 1 2 \sum_{i=1}^n \E Y_i^2 \le \sum_{i=1}^n \E |Y_i|^3 \le \delta' (V^*)^{3/2}. 
    \]
    From (\ref{eq.extremal_local_lower}) $\pr(S \in \{0, \frac 1 2\}) \ge \frac 1 {\sqrt {2 \pi V^*}} (1 - h)$. 
    Finally note that all of our used assumptions, i.e. (\ref{eq.epsilons_assumption}), $V^* \ge 83$, $\tilde{\epsilon} >  2 (V^*)^{-1/2}$ and $\pr(\bar{C}) \le \frac 1 2$ follow from (\ref{eq.clt_eps}), (\ref{eq.h_epsilon}) and (\ref{eq.deltaV}). 
\end{proofof}

\medskip

\begin{proofof}{Theorem~\ref{thm.main}}
    Let $\underbar{$c$}$, $\overline{c}$ be positive numbers, guaranteed by the norm equivalence for $\mathbb{R}^d$, such that $\underbar{$c$} \|z\| \le \|z\|_2 \le \overline{c} \|z\|$ for all $z \in \mathbb{R}^d$. Set $c_{\mathcal{S}} = \frac {\sqrt 3} 4$ if $\mathcal{S}=(\mathbb{R}^d, \|\cdot\|)$ is Hilbert and $c_{\mathcal{S}} = \frac 1 8$ otherwise. Let $\overline{C}$ be the minimum number of open sets of $l_2$-diameter one that can cover any open set of $l_2$-diameter $\overline{c}$. As such a set is contained in an open $l_2$-ball of diameter $2 \overline{c}$, the number $\overline{C}$ is finite (although exponential in $d$).
    For the Euclidean space we have $\underbar{$c$} = \overline{c} = \overline{C} = 1$ and $c_{\mathcal{S}}=\frac {\sqrt 3} 4$.
    
    Let $\concdiam_2$ be the concentration w.r.t. $\|\cdot\|_2$.
    Then $\concdiam_2(X_1 + \dots + X_n, 1) \ge \overline{C}^{-1} \concdiam(X_1 + \dots + X_n, 1)$.

    Let $c_1 = 9 c_H (d-1)^3$ be the constant from Lemma~\ref{lem.halasz_corollary2}.
    We will prove the theorem with
    \begin{equation}\label{eq.defbigC}
        C = \sqrt{2c_1 \overline{C}} \underbar{$c$}^{-1}  c_{\mathcal S}^{-1}.
    \end{equation}

    Recall that $d(x,y) = \|x-y\|$ and $d_2(x,y) = \|x-y\|_2$.
    Rescale the space by multiplying by 2, apply Lemma~\ref{lem.halasz_corollary2} with 
    $c = 2 \underbar{$c$} c_{\mathcal{S}} = 2 \sqrt{2 c_1 \overline{C}} C^{-1}$, $\delta =  \overline{C}^{-1} t(\alpha_1, \dots, \alpha_n)$ and
    \[
        \epsilon = \frac {m} {4 n} = \frac C 4 \xi_d(\bar{\alpha})^{1/2} t(\alpha_1, \dots, \alpha_n)^{-1/2} n^{-1/2},
    \]
    then rescale back by dividing by 2. We get that either $\concdiam(X_1 + \dots + X_n, 1) < t(\alpha_1, \dots, \alpha_n)$ or there is a line $L$ and vectors $a_1, \dots, a_n$ in $\mathbb{R}^d$ such that for all but at most $b$, 
    \[
        b := \frac {m} 2 = 
         \frac {C} {2} \xi_d(\bar{\alpha})^{1/2} t(\alpha_1, \dots, \alpha_n)^{-1/2} n^{1/2},
    \]
    indices in $i \in \{1,\dots, n\}$ we have $\pr(d(X_i - a_i, L) \ge c_{\mathcal{S}}) \le \pr(d_2(X_i - a_i, L) \ge \underbar{$c$} c_{\mathcal{S}}) \le \epsilon$. In the former case we are done by Lemma~\ref{lem.clt}, so let us assume the latter one. We may assume that $a_i=0$ for all $i$.
    Let $I = \{i \in \{1, \dots, n\}: \pr(d(X_i, L) \ge c_{\mathcal{S}}) \le \epsilon\}$, so $|I| \ge n - \lfloor b \rfloor$. 
    Let $\mathbb{I}_i$ be the indicator of the event $d(X_i, L) \ge c_{\mathcal{S}}$.
    It is a Bernoulli random variable
    with parameter $p_i$ and if $i \in I$ then $p_i \le \epsilon$.
    Note that (\ref{eq.m_eps}) and $c < \frac 1 3$ implies $\epsilon \le \frac 1 {60}$. 

    Define a random variable $B=\sum_{i \in I} \mathbb{I}_i$.
    We have
    \begin{align*}
        & \pr(B \ge 2 \epsilon n) 
        \le \pr(B - \E B \ge \epsilon n)
        \le \pr\left(B - \E B \ge \epsilon |I|\right)
        \le e^{-2 \epsilon^2 |I| }
        \le e^{- \frac {29 } {15} \epsilon^2 n}.
    \end{align*}
    Here
    the fourth line is a Hoeffding's concentration inequality, see, e.g. \cite{cmcd98} and
    the last line follows since by (\ref{eq.m_eps}) $|I| \ge n -  \frac {c n} {2 \cdot 5} \ge \frac {29 n} {30}$.

    For $i\in\{1,\dots,n\}$ define $\alpha_i' = \min(1,(1+\frac {60} {59} \epsilon) \alpha_i)$.
    For $i \in I$ let $X_i'$ be distributed as $X_i$ conditioned on $\mathbb{I}_i = 0$.
    Since $\epsilon \le \frac 1 {60}$ we have
    \[
        \concdiam(X_i', 1) \le  \min \left (1, \frac {\alpha_i} {1-\epsilon} \right) = \min \left (1, \alpha_i \left(1 + \frac{\epsilon} {1- \epsilon}\right) \right) \le \alpha_i'.
    \]
    Let $Y_1', \dots, Y_n'$ be independent random variables, $Y_i' \sim \nu^*_{\alpha_i'}$.
    Applying the exact result of Juškevičius~\cite{tj} for dimension 1 for any set of indices $1 \le i_1 <  \dots < i_k \le n$ we get
        $t(\alpha_{i_1}', \dots, \alpha_{i_k}') \le t(\alpha_1', \dots, \alpha_k')$.

    So by one of our main results, Theorem~\ref{close_to_line}, we have
    \begin{align}
        & \concdiam(X_1 + \dots + X_n, 1) \le  \concdiam (\sum_{i \in I} X_i',1) \nonumber
        \\ & \le \max_{I', I' \subseteq I, |I'| \ge |I| - 2 \epsilon n} \concdiam (\sum_{i \in I'} X_i',1) + \pr( B \ge 2\epsilon n) \nonumber
        \\ & \le \max_{I', I' \subseteq I, |I'| \ge |I| - 2 \epsilon n} \concdiam (\sum_{i \in I'} Y_i',1) + \pr( B \ge 2\epsilon n) \nonumber
        \\ & \le \max_{I', I' \subseteq \{1, \dots, n\}, |I'| \ge n-b -2 \epsilon n} \concdiam (\sum_{i \in I'} Y_i',1) + \pr( B \ge 2\epsilon n) \nonumber
        \\ & \le t(\alpha_1', \dots, \alpha_{\bar{n}}') + e^{-\frac {29} {15} \epsilon^2 n} \label{eq.apply_close_to_line}
    \end{align}
    where
    \[
       \bar{n} := n-\lfloor b + 2\epsilon n\rfloor = n - \lfloor m \rfloor.
    \]

    Let $\alpha \in (\frac 1 {k+1}, \frac 1 k)$. Then from (\ref{eq.V_alpha}) 
    $-\frac {\partial V_\alpha^*} {\partial \alpha} = \frac 1 {12} k (k+1) (2 k + 1)$.  For $k=1$ the right side is equal to $\frac 1 2$, while for $k\ge 2$ it is at most $6^{-1} (k^2-1) (2 k + 1)  \le 2V^*_{\alpha} (2\alpha^{-1} + 1)$. Therefore, since $V_{\alpha}$ is convex on $\alpha \in (0,1]$, for any such $\alpha$ and $h > 0$, if $\alpha(1+h) \le 1$ then
    \[
        V_{\alpha (1+h)}^* \ge V_{\alpha}^* - \alpha h \max(2 V_{\alpha}^* (2 \alpha^{-1} + 1), \frac 1 2) \ge V_{\alpha}^*(1-6h) - \frac{\alpha h} 2.
    \]
    So for $t \in \{1, \dots, n\}$
    \begin{equation}\label{eq.varepsbound}
        V'_t := \sum_{i=1}^{t}V_{\alpha_i'}^*
        \ge  V^*_{t} (1 - 7 \epsilon) - \frac {30} {59} t \bar{\alpha} \epsilon.
    \end{equation}

    By Lemma~\ref{lem.clt} and the conditions of the theorem
    \begin{align*}
       t(\alpha_1, \dots, \alpha_n) = \frac {1 \pm \epsilon'} {\sqrt {2 \pi V^*}}. 
    \end{align*}
    As $\epsilon' < \frac 1 2$ we have $t(\alpha_1, \dots, \alpha_n) \ge \frac 1 {2\sqrt{2\pi V^*}}$ and by the definition of $\epsilon$, (\ref{eq.cVstar2}) and  (\ref{eq.condition_near_one}) for any $t \ge (1-c) n$
    \[
        \frac {30 t \bar{\alpha} \epsilon} {59 V^*_t} \le   \frac {30 \cdot 4} {59 \cdot 3} \frac {\bar{\alpha} \epsilon n} {V^*} \le \frac {10 C (8 \pi)^{1/4} \sqrt{\gamma}} {59} \le \frac 1 8.
    \]
    Therefore we have
    \begin{equation}\label{eq.epsilon_V}
        \epsilon_V := 7 \epsilon + \frac {40} {59} \frac{\bar{\alpha} \epsilon n} {V^*}        \le 7 \epsilon + 0.38 C \sqrt{\gamma} 
        \le \frac 1 8 + \frac 1 8 =  \frac 1 4
    \end{equation}
    and by (\ref{eq.varepsbound}) for any $t \ge (1-c) n$
    \begin{equation}\label{eq.V_prime}
        V'_t \ge V^*_t (1 - \epsilon_V).
    \end{equation}

    We shall now check that Lemma~\ref{lem.clt} can be applied to the
    sequence $Y_1', \dots, Y_{\bar n}'$.
    Since (\ref{eq.cVstar2}) holds for $Y_1, \dots, Y_n$ and by (\ref{eq.m_eps})
    \begin{equation}\label{eq.bless}
        b + 2\epsilon n = C \xi_d(\bar{\alpha})^{\frac 1 2} t(\alpha_1, \dots, \alpha_n)^{- \frac 1 2} n^{\frac 1 2}
        \le \frac {c n} 5 < \frac {c n} 2,
    \end{equation}
    using 
 (\ref{eq.cVstar2}) and 
    (\ref{eq.V_prime})
    we get that (\ref{eq.cVstar}) also holds for $Y_1', \dots, Y'_{\bar{n}}$ with the constant $\frac c 2$ instead of $c$, i.e.:
    \[
        V'_{\lceil (1-c/2) \bar{n}\rceil} \ge (1-\epsilon_V) V^*_{\lceil (1-c/2) \bar{n}\rceil} \ge (1-\epsilon_V) \frac 3 4 V^* \ge \frac 9 {16} V^*_{\bar n} \ge \frac 9 {16} V'_{\bar n}  >  \frac 1 2 V'_{\bar{n}}. 
    \]
    The absolute moments of $\nu^*_\alpha$ decrease as $\alpha$ increase.
    So 
    since (\ref{eq.berryesseen}) holds for $Y_1, \dots, Y_n$
    we have that (\ref{eq.berryesseen}) holds also 
    for $Y_1', \dots Y_{\bar{n}}'$ with $\delta''$ instead of $\delta'$ where $\delta'' = (4/3)^{3/2} \delta' (1-\epsilon_V)^{-3/2} \le (4/3)^3 \delta'$. Here we again used (\ref{eq.cVstar2}) and (\ref{eq.V_prime}).

    We shall apply Lemma~\ref{lem.clt} to 
    the sequence
    $Y_1', \dots, Y_{\bar n}'$
    with parameters 
    \[
        (c/2, (4/3)^3 \delta', 8 \epsilon'/3)
    \]
    instead of $(c, \delta', \epsilon')$. 

    It remains to check that (\ref{eq.clt_eps}) holds for this shorter sequence. 
    By (\ref{eq.cVstar2}) and (\ref{eq.bless}) we have
    \[
        V^*_{\bar n} \ge V'_{\bar n} \ge (1 - \epsilon_V) V_{\bar n}^* \ge \frac 3 4 (1-\epsilon_V) V^* \ge \frac 9 {16} V^*.
    \]
    Thus 
    replacing $\delta'$ with $(4/3)^3 \delta'$ and replacing $c$ with $\frac c 2$ increases the left side of (\ref{eq.clt_eps}) by 
    $2^{3/4}(4/3)^{3/2} \le \frac 8 3$,
    so (\ref{eq.clt_eps}) holds 
for 
 $Y_1', \dots, Y_{\bar n}'$
with the claimed 
parameters by our stronger requirement that 
$\epsilon' \le \frac 3 {16}$.

    Thus Lemma~\ref{lem.clt} yields:
    \begin{align*}
        t(\alpha_1', \dots, \alpha_{\bar{n}}') &\le \frac {1 +  \frac 8 3   \epsilon'} {\sqrt {2 \pi V_{\bar{n}}^* (1 - \epsilon_V)}}
        \\ & \le \frac 1 {\sqrt{2 \pi V_{\bar{n}}^*}} (1 + \frac {2 \cdot 8} 3 \epsilon' + 2 \epsilon_V)
        \\ & \le \frac 1 {\sqrt{2 \pi V_{\bar{n}}^*}} (1 + 6 \epsilon' + 14 \epsilon + C \sqrt{\gamma}).
    \end{align*}  
    Here in the second line we used $\frac 8 3  \epsilon'\le \frac 1 2$ and $\epsilon_V \le \frac 1 4$ and in the last line we used  (\ref{eq.epsilon_V}). The proof is completed by combining the last bound with (\ref{eq.apply_close_to_line}) and using $\epsilon = \frac m {4n}$.
\end{proofof}

\medskip

Now our first two theorems follow easily as simple special cases.

\medskip

\begin{proofof}{Theorem~\ref{thm.front} and Theorem~\ref{teor1}}
 It is enough to prove the more general Theorem~\ref{teor1}. We may assume $\alpha < 1$. Let us apply Theorem~\ref{thm.main} to each $n \in \{1, 2, \dots\}$ with $\alpha_1 = \dots = \alpha_n = \alpha$.

       For the i.i.d. summands $Y_1, \dots, Y_n \sim \nuopt \alpha$ we have $V^*_{\lfloor (1-c)n \rfloor} \ge (1-c) V^*_n$ for any $c \in (0,1)$. So we can take, for example $c=\frac 1 4$ so that (\ref{eq.cVstar2}) is satisfied for any $n$. Let the constant $C = C(d, \|\cdot\|)$ be as in (\ref{eq.defbigC}). 
    Since $\alpha$ is fixed, (\ref{eq.berryesseen}) holds for any positive $\delta'$
    with $n$ sufficiently large. Since $V^*_n \to \infty$ as $n \to \infty$,
    we can fix a $\delta' \in (0,1)$ such that for all $n$ large enough,
    $\epsilon'$ defined in (\ref{eq.clt_eps}) satisfies $\epsilon' \le 0.05$.
    We can actually make $\epsilon'$ arbitrarily small for all large enough $n$ if we choose $\delta'$ small enough.

    As $\bar{\alpha}$ ($\bar{\alpha}=\alpha$) is constant and $V^*_n$ is linear in $n$, we can choose
    an arbitrarily small $\gamma$ such that (\ref{eq.condition_near_one})
    is satisfied for all $n$ large enough.

    Consequently, it follows by Lemma~\ref{lem.clt} (or a standard local limit theorem) that $t(\alpha_1, \dots, \alpha_n) =  \concdiam(Y_1 + \dots +Y_n, 1) = (2 \pi n V_\alpha)^{-\frac 1 2} (1 + o(1))$, so $m$ is of the order $n^{\frac 3 4}$. Therefore (\ref{eq.m_eps}) is satisfied for all $n$ large enough. This also implies that for the error term we have $e^{-m^2/(9n)} = e^{-\Omega(n^{\frac 1 2})} = o(t(\alpha_1, \dots, \alpha_n))$.

    As $m n^{-1}=o(1)$, $V^*_{n-\lfloor m \rfloor} = V^*_n (1+o(1))$ and Theorem~\ref{thm.main} holds with arbitrarily small $\epsilon'$ and $\gamma$,
    for all $n$ large enough, we get
    \begin{align*}
       &\concdiam(X_1 + \dots + X_n, 1) \le \frac 1 {\sqrt{2 \pi V_n^*}} (1 + o(1)) 
        \\ &
        = \pr(Y_1 + \dots + Y_n \in \{0,\frac 1 2\}) (1+o(1))
       \\ &
        = \concdiam(Y_1 + \dots + Y_n, 1) (1 + o(1)).
    \end{align*}
    In the setting of Theorem~\ref{thm.front} just note that the value of the concentration function is not affected if
    we shift each $Y_i$ by $\frac {k+1} 2$ to make it distributed as $U_i$.
\end{proofof}

Finally, Theorem~\ref{local} follows as an asymptotic version of Theorem~\ref{thm.main}.

\medskip

\begin{proofof}{Theorem~\ref{local}}
    By the assumptions of the theorem, for each $c \in (0, \frac 1 3)$
    we can choose sequences of positive numbers $\{\delta'_l\}$, $\{\gamma_l\}$
    $\delta'_l \to 0$, $\gamma_l \to 0$ so that for $l \in \{1,2,\dots\}$
    Theorem~\ref{thm.main} applies with the absolute constant $C=C(d, \|\cdot\|)$ defined in (\ref{eq.defbigC}), 
    $\delta' = \delta_l'$ and $\gamma = \gamma_l'$.
    With this choice the left side of (\ref{eq.clt_eps}) converges to 0 as $l \to \infty$, 
    so we get that $\concdiam(X_{l,1} + \dots + X_{l,n}, 1) \le (2 \pi V^*_{l, \lceil n_l(1-c) \rceil})^{-1/2} (1 + o(1)) + r_l$,
    where $r_l = \exp(-C^2 \xi_d(\bar{\alpha}_l) t_l^{-1}/9)$ and $t_l=t(\alpha_{l,1}, \dots, \alpha_{l,n_l})$. Assuming that 
    \begin{equation}\label{eq.exp_remainder}
        r_l = o((2 \pi V^*_{l, {n_l}})^{-1/2})
    \end{equation}
    as $l\to\infty$, the proof of the upper bound is 
    completed by considering $c$ arbitrarily small, using the last assumption of the theorem and using Remark~\ref{rmk.explicit_bound}.

    By Lemma~\ref{lem.clt} we have that 
    \begin{equation}\label{eq.tl0}
        t_l =  (2 \pi V^*_{l, n_l})^{-1/2} (1 + o(1)) \to 0.
    \end{equation}
    This implies the lower bound in (\ref{eq.sharpb}).

    It remains to show (\ref{eq.exp_remainder}).
    First consider the case $\xi_d(\bar{\alpha}_l)=\Omega(1)$. 
    Then by (\ref{eq.tl0}) $r_l = \exp(-\Omega(t_l^{-1})) = o(t_l^{-1})$
    and this implies (\ref{eq.exp_remainder}).

    Now consider the case $\bar{\alpha}_l \to 0$. We can assume $\bar{\alpha}_l \le \frac 1 4$.
    Using (\ref{eq.tl0}) we can see that (\ref{eq.exp_remainder})
    holds if, for example,
    \[
        \frac {C^2 \bar{\alpha}_l} 9 \ge \frac {2 \ln \left( (2 \pi V^*_{l, n_l})^{1/2} \right)} {(2 \pi V^*_{l, n_l})^{1/2}}
    \]
    which for large enough $l$ follows from
    \[
        \frac {C^2 \bar{\alpha}_l} 9 \ge \frac {2 \ln  V^*_{l, n_l}} {(2 \pi V^*_{l, n_l})^{1/2}}.
    \]
    Remark~\ref{rmk.explicit_bound} yields for $\bar{\alpha}_l \le \frac 1 2$
    \[
        \bar{\alpha}_l \ge \left(\frac 3 4 \cdot \frac {\pi n_l} 6 \right)^{\frac 1 2} ({2 \pi V^*_{l, n_l}})^{- \frac 1 2} \ge \frac 1 2 n_l^{\frac 1 2} ({2 \pi V^*_{l, n_l}})^{- \frac 1 2} 
    \]
    Thus a sufficient condition for (\ref{eq.exp_remainder}) is
    \[
        \frac {C^2 n_l^{\frac 1 2}} {9 \cdot 2 (2 \pi V^*_{l, n_l})^{1/2} }  \ge \frac {2 \ln V^*_{l, n_l}} {(2 \pi V^*_{l, n_l})^{1/2}},
    \]
    \[
        \frac {C^2 n_l^{\frac 1 2}} {18}  \ge 2 \ln  V^*_{l, n_l}\quad \mbox{or}\quad V^*_{l,n_l} \le \exp(\frac 1 {36} C^2 n_l^{\frac 1 2}).
    \]
\end{proofof}


%
%
%

\begin{proofof}{Remark~\ref{rmk.cormain}}
    The proof follows by Theorem~\ref{local}, the fact that
    $V_{\bar{\alpha}_l} = \frac 1 {12 \bar{\alpha}^2_l} (1 + o(1))$ when $\bar{\alpha}_l \to 0$ and $V_{\bar{\alpha}_l} = \frac 1 2 (1-\bar{\alpha}_l)$ when $\bar{\alpha}_l > 0.5$. 
\end{proofof}

\medskip

Let us end this section with a comparison of our results to a known result. For the case $d=1$, Kesten~\cite{kesten} proved a concentration bound of the correct order. 

\begin{theorem}(Theorem~2 of~\cite{kesten}) 
    There is an absolute constant $C > 0$ such that for independent random variables $X_1, \dots, X_n$ and any  $0<\lambda_1, \dots, \lambda_n \le 2L$ if
$S_n = X_1 + \dots + X_n$ then
\begin{align*}
    \conc(S_n, L) 
    & \le 4 \sqrt 2 (1 + 9 C) L \frac{\sum_{i=1}^n \lambda_i^2 (1-\conc(X_i, \lambda_i))\conc(X_i, L)}{\left(\sum_{i=1}^n \lambda_i^2 (1 - \conc(X_i, \lambda_i))\right)^{\frac 3 2}}
    \\ &
    \le 4 \sqrt 2 (1 + 9 C) L \frac{\max_{i} \conc(X_i, L)}{\sqrt{\sum_{i=1}^n \lambda_i^2 (1 - \conc(X_i, \lambda_i))}}.
\end{align*}
\end{theorem}
Applying it with our conditions $L = \lambda_1 = \dots = \lambda_n = 1$ and, for example, with $\conc(X_i, 1) = \bar{\alpha}=\bar{\alpha}(n)$ for all $i$ we get
\[
    \conc(X_1 + \dots + X_n, 1) 
    \le 4 \sqrt 2 (1 + 9 C)  \frac{\bar{\alpha}}{\sqrt{(1 - \bar{\alpha})n}}.
\]
For $\bar{\alpha} \ge \frac 1 2$ and $1 - \bar{\alpha} \gg n^{-1/3}$, Theorem~\ref{local} and (\ref{eq.V_alpha}) show that $\conc(X_1 + \dots + X_n, 1) \le \frac 1 {\sqrt {{\pi (1-\bar{\alpha})n}}} (1+o(1))$. Depending on $\bar{\alpha}$, Kesten's constant is at least 5--10 times larger. For $\bar{\alpha} \in (0, \frac 1 2)$, $\bar{\alpha} \gg n^{-1/2}$
the constant in our cruder bound (\ref{eq.crudeb}) is at least 4--5 times smaller than Kesten's. For the case $d=1$, using Remark~\ref{rmk.explicit_bound}, the precise asymptotics also follow from Theorem~1.2 of Juškevičius~\cite{tj} without the restrictions on $\bar{\alpha}$ -- many of them in Theorem~\ref{local} seem to be artefacts of our method and are probably not necessary.

\bigskip

\textbf{Acknowledgment.} We thank V.K.'s brother Vilius Kurauskas for drawing attention to a Quanta Magazine article that inspired us to consider perfect graphs in this work.

\bigskip

\textsc{Institute of Computer Science of the Czech Academy of Sciences, Pod Vodárenskou věží 271/2
182 00 Praha 8, Czechia}

Email address: tomas.juskevicius@gmail.com

\bigskip

\hyphenation{a-ka-de-mi-jos}

\textsc{Faculty of Mathematics and Informatics, Vilnius University, Akademijos 4, 08412, Vilnius, Lithuania}

Email address: valentas@gmail.com

\end{document}